\documentclass[12pt,a4paper,oneside,reqno,notitlepage]{amsart}
\usepackage{amssymb}
\textwidth
150mm

\newtheorem{theorem}{Theorem}
\theoremstyle{plain}

\newtheorem{corollary}{Corollary}

\newtheorem{definition}{Definition}

\newtheorem{lemma}{Lemma}

\numberwithin{equation}{section}

\begin{document}

\baselineskip 8mm
\parindent 9mm

\title[]
{Bound state solutions for non-autonomous fractional Schr\"{o}dinger-Poisson equations with critical exponent}

\author{Kexue  Li}

\address{Kexue Li\newline
School of Mathematics and Statistics, Xi'an Jiaotong University, Xi'an 710049, China}
\email{kxli@mail.xjtu.edu.cn}

\thanks{{\it 2010 Mathematics Subjects Classification}: 35Q40, 58E30}

\begin{abstract}
In this paper, we study the  fractional Schr\"{o}dinger-Poisson equation
\begin{equation*}
 \ \left\{\begin{aligned}
&(-\Delta)^{s}u+V(x)u+K(x)\phi u=|u|^{2^{\ast}_{s}-2}u, &\mbox{in} \ \mathbb{R}^{3},\\
&(-\Delta)^{s}\phi=K(x)u^{2},&\mbox{in} \ \mathbb{R}^{3},
\end{aligned}\right.
\end{equation*}
where $s\in (\frac{3}{4},1]$, $2^{\ast}_{s}=\frac{6}{3-2s}$ is the fractional critical exponent, $K\in L^{\frac{6}{6s-3}}(\mathbb{R}^{3})$ and $V\in L^{\frac{3}{2s}}(\mathbb{R}^{3})$ are nonnegative functions. If $\|V\|_{\frac{3}{2s}}+\|K\|_{\frac{6}{6s-3}}$ is sufficiently small,
we prove that the equation has at least one bound state solution.
\end{abstract}
\maketitle

\section{\textbf{Introduction}}
In this paper, we are concerned with the existence of bound state solutions for the following fractional Schr\"{o}dinger-Poisson equation
\begin{equation}\label{fsp}
\ \left\{\begin{aligned}
&(-\Delta)^{s}u+V(x)u+K(x)\phi u=|u|^{2^{\ast-2}_{s}}u, &\mbox{in} \ \mathbb{R}^{3},\\
&(-\Delta)^{s}\phi=K(x)u^{2}, &\mbox{in} \ \mathbb{R}^{3},
\end{aligned}\right.
\end{equation}
where $s\in (\frac{3}{4},1]$, $2^{\ast}_{s}=\frac{6}{3-2s}$ is the fractional critical exponent, $K\in L^{\frac{6}{6s-3}}(\mathbb{R}^{3})$ and $V\in L^{\frac{3}{2s}}(\mathbb{R}^{3})$ are nonnegative functions, $(-\Delta)^{s}$ is the fractional Laplacian defined by
\begin{align}\label{fractionalLaplacian}
(-\triangle)^{s}u(x)=c_{s}P.V.\int_{\mathbb{R}^{3}}\frac{u(x)-u(y)}{|x-y|^{3+2s}}dy,
\end{align}
where
\begin{align*}
c_{s}=2^{2s}\pi^{-\frac{3}{2}}\frac{\Gamma(\frac{3+2s}{2})}{|\Gamma(-s)|}.
\end{align*}

When $s=1$, (\ref{fsp}) is the classical Schr\"{o}dinger-Poisson equation or the more general one
\begin{equation}\label{SP}
\ \left\{\begin{aligned}
&-\Delta u+V(x)u+K(x)\phi u=f(x,u), &\mbox{in} \ \mathbb{R}^{3},\\
&-\Delta \phi=K(x)u^{2}, &\mbox{in} \ \mathbb{R}^{3},
\end{aligned}\right.
\end{equation}
The Schr\"{o}dinger-Poisson system has been introduced in \cite{Benci} as a model describing a quantum particle interacting with a electromagnetic field. In recent years, (\ref{SP}) has attracted much attention,
we refer to \cite{Aprile,Ambrosetti,Jiang,Sun,Cerami,Y,Z,He} and the references therein.

When $V(x)=1$, $K(x)=\lambda$ and $f(x,u)=u^{p}$, Ruiz \cite{Ruiz} studied the problem
\begin{equation*}
\ \left\{\begin{aligned}
&-\Delta u+u+\lambda\phi u=u^{p}, \\
&-\Delta \phi=u^{2}, \lim_{|x|\rightarrow +\infty}\phi(x)=0,
\end{aligned}\right.
\end{equation*}
where $u,\phi:\mathbb{R}^{3}\rightarrow \mathbb{R}$ are positive radial functions, $\lambda>0$ and $1<p<5$. The existence and nonexistence are established with different parameters $p$ and $\lambda$. In the case that $p\in (2,3)$, Ruiz developed a new approach consists of minimizing the associated energy functional on a certain manifold which is a combination of Nehari manifold and Pohozaev equality. In the case that $p\in (1,2)$ and $\lambda$ is small enough, the existence results are obtained. Liu and Guo \cite{Liu} considered the following Schr\"{o}dinger-Poisson equation with critical growth
\begin{equation*}
\ \left\{\begin{aligned}
&-\Delta u+V(x)u+\lambda\phi u=\mu|u|^{q-1}u+u^{5}, \\
&-\Delta \phi=u^{2},
\end{aligned}\right.
\end{equation*}
where $q\in (2,5)$, $\lambda>0$. Under some assumptions on the potential $V$, they used the variational methods to prove the existence of positive ground state solutions. In \cite{ZMX}, the authors considered (\ref{SP}) with $f(u)=u^{5}$ and applied a linking theorem to prove the existence of bound state solutions when $V(x)$, $K(x)$ satisfy some conditions. If $V(x)=1$, $f(x,u)=a(x)|u|^{p-2}+u^{5}$, $p\in (4,6)$, when $K(x)$ and $a(x)$ satisfy some assumptions, Zhang \cite{ZhangJian}  studied the existence of ground state solutions and nodal solutions for (\ref{SP}).

If $\phi=0$, (\ref{fsp}) reduces to a fractional Schr\"{o}dinger equation, which is a fundmental equation in fractional quantum mechanics \cite{Laskin1,Laskin2}. In fact, if one extends the Feynman path integral from the Brownian-like to L\'{e}vy-like quantum mechanical paths, the classical Schr\"{o}dinger equation will change into
the fractional Schr\"{o}dinger equation. In the last decade, the existence of solutions for fractional Schr\"{o}dinger equations has been investigated by many authors, we refer to \cite{Ionescu,Lemm,Ambrosio,Bieganowski,HZ,Ao,Secchi}.

To the best of our knowledge, there are few papers which considered the equation (\ref{fsp}). In \cite{Zhang}, the authors considered a fractional Schr\"{o}dinger-Poisson system
\begin{equation*}
\ \left\{\begin{aligned}
&(-\Delta)^{s}u+\lambda\phi u=g(u), &\mbox{in} \ \mathbb{R}^{3},\\
&(-\Delta)^{t}\phi=\lambda u^{2},&\mbox{in} \ \mathbb{R}^{3},
\end{aligned}\right.
\end{equation*}
where $g(u)$ satisfies the Berestucki-Lions conditions, the existence of positive solutions is proved.
Teng \cite{T} studied the existence of ground state solutions for the nonlinear fractional Schr\"{o}dinger-Poisson equation with critical Sobolev exponent
\begin{equation*}
\ \left\{\begin{aligned}
&(-\Delta)^{s}u+V(x)u+\phi u=\mu|u|^{p-1}+|u|^{2_{s}^{\ast}-2}u, &\mbox{in} \ \mathbb{R}^{3},\\
&(-\Delta)^{t}\phi=u^{2},&\mbox{in} \ \mathbb{R}^{3},
\end{aligned}\right.
\end{equation*}
where $\mu>0$ is a parameter, $1<p<2_{s}^{\ast}-1=\frac{3+2s}{3-2s}$, $s,t\in (0,1)$ and $2s+2t>3$, the author prove the existence of a nontrivial ground state solutions by using the method of
Nehari-Pohozaev manifold and the arguments of Brezis-Nirenberg, the monotonic trick and global compactness Lemma. Shen and Yao \cite{Shen} applied Nehari-Pohozaev type
manifold to prove the existence of a nontrivial least energy solution for the nonlinear fractional Schr\"{o}dinger-Poisson system
\begin{equation*}
\ \left\{\begin{aligned}
&(-\Delta)^{s}u+V(x)u+\phi u=\mu|u|^{p-1}, &\mbox{in} \ \mathbb{R}^{3},\\
&(-\Delta)^{t}\phi=u^{2},&\mbox{in} \ \mathbb{R}^{3},
\end{aligned}\right.
\end{equation*}
where $s,t\in (0,1)$, $s<t$ and $2s+2t>3$, $2<p<\frac{3+2s}{3-2s}$. In \cite{Yu}, the authors studied the fractional Schr\"{o}dinger-Poisson system
\begin{equation*}
\ \left\{\begin{aligned}
&\varepsilon^{2s}(-\Delta)^{s}u+V(x)u+\phi u=K(x)|u|^{p-2}u, &\mbox{in} \ \mathbb{R}^{3},\\
&\varepsilon^{2s}(-\Delta)^{s}\phi=u^{2},&\mbox{in} \ \mathbb{R}^{3},
\end{aligned}\right.
\end{equation*}
where $\varepsilon>0$, $\frac{3}{4}<s<1$, $4<p<\frac{6}{3-2s}$, $V(x)\in C(\mathbb{R}^{3})\cap L^{\infty}(\mathbb{R}^{3})$ is has positive global minimim, $K(x)\in C(\mathbb{R}^{3})\cap L^{\infty}(\mathbb{R}^{3})$ is positive and has global maximum. They proved the existence of positive ground solution and determined a concrete set related to the
potentials $V$ and $K$ as the concentration position of these ground state solutions as $\varepsilon\rightarrow 0$.

In this paper, we are devoted to establishing the existence of bounded state solutions for the fractional Schr\"{o}dinger-Poisson equation (\ref{fsp}). We consider the nonlinear term $f(x,u)=|u|^{2_{s}^{\ast}-2}$ without the
subcritical part and we deal with the problem (\ref{fsp}) in $D^{s,2}(\mathbb{R}^{3})$ instead of in $H^{s}(\mathbb{R}^{3})$. Under some assumptions on the potentials $V$ and $K$, we prove that (\ref{fsp}) can not be solved by
constrained minimization on the Nehari manifold. After we prove some concentration-compactness results, the existence of a bounded state solution is obtained by linking theorem.

Our main result is as folows:
\begin{theorem}\label{boundstate}
Assume that  $V\geq 0$, $V\in L^{\frac{3}{2s}}(\mathbb{R}^{3})$, $K\geq 0$, $K\in L^{\frac{6}{6s-3}}(\mathbb{R}^{3})$ and
\begin{align}\label{svk}
0<S_{s}^{-1}\|V\|_{\frac{3}{2s}}+S_{s}^{\frac{3}{2s}-3}\|K\|_{\frac{6}{6s-3}}^{2}\leq 2^{\frac{2_{s}^{\ast}-4}{2_{s}\ast}}-1,
\end{align}
where
\begin{align*}
S_{s}=\inf_{u\in D^{s,2}(\mathbb{R}^{3})\backslash\{0\}}\frac{\int_{\mathbb{R}^{3}}|(-\Delta)^{\frac{s}{2}}u|^{2}dx}{(\int_{\mathbb{R}^{3}}|u|^{2_{s}^{\ast}})^{\frac{2_{s}^{\ast}}{2}}}
\end{align*}
is the best Sobolev constant for the embedding $D^{s,2}(\mathbb{R}^{3})\rightarrow L^{2_{s}^{\ast}}(\mathbb{R}^{3})$,
then the problem (\ref{fsp}) has at least one bound state solution.
\end{theorem}

This paper is organized as follows. In Section 2, we first present the variational setting of the problem, then we prove some lemmas. In Section 3, we study the Palais-Smale sequence and get a compactness theorem. In Section 4,
we prove the existence of bounded state solutions by linking theorem.

\section{\textbf{Variational setting and preliminaries}}
For $p\in [1,\infty)$, we denote by $L^{p}(\mathbb{R}^{3})$ the usual Lebesgue space with the norm $\|u\|_{p}=\left(\int_{\mathbb{R}^{3}}|u(x)|^{p}dx\right)^{\frac{1}{p}}$. For any $r>0$ and for any $x\in \mathbb{R}^{3}$,
$B_{r}(x)$ denotes the ball of radius $r$ centered at $x$.
For any $s\in (0,1)$, we recall some definitions of fractional Sobolev spaces $H^{s}(\mathbb{R}^{3})$ and the fractional Laplacian $(-\Delta)^{s}$, for more details, we refer to \cite{DPV}.  $H^{s}(\mathbb{R}^{3})$ is defined as follows
\begin{align*}
H^{s}(\mathbb{R}^{3})=\left\{u\in L^{2}(\mathbb{R}^{3}): \int_{\mathbb{R}^{3}}(1+|\xi|^{2s})|\mathcal{F}u(\xi)|^{2}d\xi<\infty\right\}
\end{align*}
with the norm
\begin{align}\label{alternative}
\|u\|_{H^{s}}=\left(\int_{\mathbb{R}^{3}}(|\mathcal{F}u(\xi)|^{2}+|\xi|^{2s}|\mathcal{F}u(\xi)|^{2})d\xi\right)^{\frac{1}{2}},
\end{align}
where $\mathcal{F}u$ denotes the Fourier transform of $u$. As usual $H^{-s}(\mathbb{R}^{3})$ denotes the dual of $H^{s}(\mathbb{R}^{3})$.
By $\mathcal{S}(\mathbb{R}^{n})$, we denote the Schwartz space of rapidly decaying $C^{\infty}$ functions in $\mathbb{R}^{n}$. For $u\in \mathcal{S}(\mathbb{R}^{n})$ and $s\in (0,1)$, $(-\Delta)^{s}$ is defined by
\begin{align}\label{Fourier}
(-\Delta)^{s}f=\mathcal{F}^{-1}(|\xi|^{2s}(\mathcal{F}f)), \ \forall \xi\in \mathbb{R}^{n}.
\end{align}
In fact, (\ref{Fourier}) is equivalent to (\ref{fractionalLaplacian}), see \cite{Tzirakis}. By Plancherel's theorem, we have $\|\mathcal{F}u\|_{2}=\|u\|_{2}$, $\||\xi|^{s}\mathcal{F}u\|_{2}=\|(-\Delta)^{\frac{s}{2}}u\|$. Then by (\ref{alternative}), we get the equivalent norm
\begin{align*}
\|u\|_{H^{s}}=\left(\int_{\mathbb{R}^{3}}(|(-\Delta)^{\frac{s}{2}}u|^{2}+|u|^{2})dx\right)^{\frac{1}{2}}.
\end{align*}
 From Theorem 6.5 and Corollary 7.2 in \cite{DPV}, it is known that the space $H^{s}(\mathbb{R}^{3})$  is continuously embedded in $L^{q}(\mathbb{R}^{3})$ for any $q\in [2,2^{\ast}_{s}]$
 and the embedding $H^{s}(\mathbb{R}^{3})\hookrightarrow L^{q}(\mathbb{R}^{3})$ is locally compact for $q\in [1, 2^{\ast}_{s})$, where $2^{\ast}_{s}=\frac{6}{3-2s}$ is the fractional critical Sobolev exponent. \\
For $s\in (0,1)$, the fractional Sobolev space $D^{s,2}(\mathbb{R}^{3})$ is defined as follows
\begin{align*}
D^{s,2}(\mathbb{R}^{3})=\left\{u\in L^{2^{\ast}_{s}}(\mathbb{R}^{3}): |\xi|^{s}\mathcal{F}u(\xi)\in L^{2}(\mathbb{R}^{3})\right\},
\end{align*}
which is the completion of $C^{\infty}_{0}(\mathbb{R}^{3})$ with respect to the norm
\begin{align*}
\|u\|_{D^{s,2}}=\left(\int_{\mathbb{R}^{3}}|(-\Delta)^{\frac{s}{2}}u|^{2}dx\right)^{\frac{1}{2}}=\left(\int_{\mathbb{R}^{3}}|\xi|^{2s}|\mathcal{F}u(\xi)|^{2}d\xi\right)^{\frac{1}{2}}.
\end{align*}
The best Sobolev constant is defined by
\begin{align}\label{minimizer}
S_{s}=\inf_{u\in D^{s,2}(\mathbb{R}^{3})\backslash \{0\}}\frac{\int_{\mathbb{R}^{3}}|(-\Delta)^{\frac{s}{2}}u|^{2}dx}{\left(\int_{\mathbb{R}^{3}}|u|^{2^{\ast}_{s}}dx\right)^{\frac{2}{2^{\ast}_{s}}}}.
\end{align}
By $C$, we denote the generic constants, which may change from line to line. We consider the variational setting of (\ref{fsp}).
For $u\in D^{s,2}(\mathbb{R}^{3})$, define the linear operator $T_{u}: D^{s,2}(\mathbb{R}^{3})\rightarrow \mathbb{R}$ as
\begin{align*}
T_{u}(v)=\int_{\mathbb{R}^{3}}K(x)u^{2}vdx.
\end{align*}
For $K\in L^{\frac{6}{6s-3}}(\mathbb{R}^{3})$, by the H\"{o}lder inequality, we have
\begin{align}\label{dt2}
|T_{u}(v)|\leq \|K\|_{\frac{6}{6s-3}}\|u\|^{2}_{\frac{6}{3-2s}}\|v\|_{2^{\ast}_{s}}\leq \|K\|_{\frac{6}{6s-3}}S^{-\frac{3}{2}}_{s}\|u\|_{D^{s,2}}^{2}\|v\|_{D^{s,2}}.
\end{align}

Set
\begin{align*}
\eta(a,b)=\int_{\mathbb{R}^{3}}(-\Delta)^{\frac{s}{2}}a\cdot (-\Delta)^{\frac{s}{2}}b\ dx, \ a,b\in D^{s,2}(\mathbb{R}^{3}).
\end{align*}
It is clear that $\eta(a,b)$ is bilinear, bounded and coercive. The Lax-Milgram theorem implies that for every $u\in D^{s,2}(\mathbb{R}^{3})$, there exists a unique $\phi^{s}_{u}\in D^{s,2}(\mathbb{R}^{3})$ such that $T_{u}(v)=\eta(\phi^{s}_{u},v)$ for any $v\in D^{s,2}(\mathbb{R}^{3})$, that is
\begin{align}\label{dt3}
\int_{\mathbb{R}^{3}}(-\Delta)^{\frac{s}{2}}\phi^{s}_{u}(-\Delta)^{\frac{s}{2}}vdx=\int_{\mathbb{R}^{3}}K(x)u^{2}vdx.
\end{align}
Therefore, $(-\Delta)^{s}\phi^{s}_{u}=K(x)u^{2}$ in a weak sense.
For $x\in \mathbb{R}^{3}$, we have
\begin{align}\label{Riesz}
\phi^{s}_{u}(x)=c_{s}\int_{\mathbb{R}^{3}}\frac{K(y)u^{2}(y)}{|x-y|^{3-2s}}dy,
\end{align}
which is the Riesz potential \cite{Stein}, where
\begin{align*}
c_{s}=\frac{\Gamma(\frac{3-2s}{2})}{\pi^{3/2}2^{2s}\Gamma(s)}.
\end{align*}
Moreover,
\begin{align}\label{phisu}
\|\phi^{s}_{u}\|_{D^{s,2}}=\|T_{u}\|_{\mathcal{L}(D^{s,2},\mathbb{R})}.
\end{align}
Since $K\in L^{\frac{6}{6s-3}}(\mathbb{R}^{3})$, by (\ref{dt3}), (\ref{dt2}),
\begin{align}\label{phi2}
\|\phi^{s}_{u}\|^{2}_{D^{s,2}}&=\int_{\mathbb{R}^{3}}K(x)\phi^{s}_{u}u^{2}dx\leq \|K\|_{\frac{6}{6s-3}}\|u\|^{2}_{\frac{6}{3-2s}}\|\phi^{s}_{u}\|_{2^{\ast}_{s}}\nonumber\\
&\leq S^{-\frac{1}{2}}_{s}\|K\|_{\frac{6}{6s-3}}\|u\|^{2}_{\frac{6}{3-2s}}\|\phi^{s}_{u}\|_{D^{s,2}}.
\end{align}
Then
\begin{align}\label{ph}
\|\phi^{s}_{u}\|_{D^{s,2}}\leq S^{-\frac{1}{2}}_{s}\|K\|_{\frac{6}{6s-3}}\|u\|^{2}_{\frac{6}{3-2s}}\leq CS^{-\frac{1}{2}}_{s}\|K\|_{\frac{6}{6s-3}}\|u\|_{H^{s}}.
\end{align}
Substituting $\phi^{s}_{u}$ in (\ref{fsp}), we have the fractional Schr$\ddot{\mbox{o}}$dinger equation
\begin{align}\label{fractional Schrodinger}
(-\Delta)^{s}u+V(x)u+K(x)\phi^{s}_{u}u=|u|^{2^{\ast}_{s}-2}u, \ x\in \mathbb{R}^{3},
\end{align}
The energy functional $I: D^{s,2}(\mathbb{R}^{3})\rightarrow \mathbb{R}$ corresponding to problem (\ref{fractional Schrodinger}) is defined by
\begin{align}\label{corresponding}
I(u)=\frac{1}{2}\int_{\mathbb{R}^{3}}(|(-\Delta)^{\frac{s}{2}}u|^{2}+V(x)u^{2})dx+\frac{1}{4}\int_{\mathbb{R}^{3}}K(x)\phi^{s}_{u}u^{2}dx-\frac{1}{2^{\ast}_{s}}\int_{\mathbb{R}^{3}}|u|^{2^{\ast}_{s}}dx.
\end{align}
It is easy to see that $I$ is well defined in $D^{s,2}(\mathbb{R}^{3})$ and $I\in C^{1}(D^{s,2}(\mathbb{R}^{3}),\mathbb{R})$, and
\begin{align*}
\langle I'(u),v\rangle=\int_{\mathbb{R}^{3}}\left((-\Delta)^{\frac{s}{2}}u(-\Delta)^{\frac{s}{2}}v+V(x)uv+K(x)\phi^{s}_{u}uv-|u|^{2^{\ast}_{s}-2}uv\right)dx,
\end{align*}
for any $v\in D^{s,2}(\mathbb{R}^{3})$.
\begin{definition}
(1) We call $(u,\phi)\in D^{s,2}(\mathbb{R}^{3})\times D^{s,2}(\mathbb{R}^{3})$ is a weak solution of (\ref{fsp}) if $u$ is a weak solution of (\ref{fractional Schrodinger}).\\
(2)We call  $u$ is a weak solution of (\ref{fractional Schrodinger}) if
\begin{align*}
\int_{\mathbb{R}^{3}}\left((-\Delta)^{\frac{s}{2}}u(-\Delta)^{\frac{s}{2}}v+V(x)uv+K(x)\phi^{s}_{u}uv-|u|^{2^{\ast}_{s}-2}uv\right)dx=0,
\end{align*}
for any $v\in D^{s,2}(\mathbb{R}^{3})$.
\end{definition}

Let $\Phi: D^{s,2}(\mathbb{R}^{3})\rightarrow D^{s,2}(\mathbb{R}^{3})$ be the operator
\begin{align*}
\Phi(u)=\phi^{s}_{u},
\end{align*}
and set
\begin{align*}
N(u)=\int_{\mathbb{R}^{3}}K(x)\phi^{s}_{u}u^{2}dx.
\end{align*}
\begin{lemma}\label{Hardy}
Let $u_{n}\in D^{s,2}(\mathbb{R}^{3})$ be such that $u_{n}\rightharpoonup 0$ in $D^{s,2}(\mathbb{R}^{3})$. Then, up to subsequences, $u_{n}\rightarrow 0$ in $L_{loc}^{p}(\mathbb{R}^{3})$ for any $p\in [2,2_{s}^{\ast})$.
\end{lemma}
\begin{proof}
By the fractional Hardy inequality (see (1.1) in \cite{Yafaev}),
\begin{align*}
\int_{\mathbb{R}^{3}}\frac{u_{n}^{2}}{(1+|x|)^{2s}}dx\leq c_{s}\int_{\mathbb{R}^{3}}|(-\Delta)^{\frac{s}{2}}u_{n}|^{2}dx\leq C.
\end{align*}
Let $r>0$ and let $A\subset \mathbb{R}^{3}$ be such that $A\subset B_{r}(0)$. We have
\begin{align*}
\int_{A}u_{n}^{2}dx\leq (1+r)^{2s}\int_{B_{r}(0)}\frac{u_{n}^{2}}{(1+|x|)^{2s}}dx\leq \widetilde{C}.
\end{align*}
this together with $u_{n}\in D^{s,2}(\mathbb{R}^{3})$ yield that $u_{n}$ is bounded in $H^{s}(A)$. Since $u_{n}\rightharpoonup 0$ in $D^{s,2}(\mathbb{R}^{3})$, then up to a subsequence, $u_{n}\rightarrow 0$ in $L^{2}(A)$.
By $L^{p}$ interpolation inequality, for all $p\in [2,2_{s}^{\ast})$,
\begin{align*}
\left(\int_{A}|u_{n}|^{p}dx\right)^{\frac{1}{p}}\leq \left(\int_{A}|u_{n}|^{2}dx\right)^{\frac{\alpha}{2}}\left(\int_{A}|u_{n}|^{2_{s}^{\ast}}dx\right)^{\frac{1-\alpha}{2_{s}^{\ast}}},
\end{align*}
where $\frac{\alpha}{2}+\frac{1-\alpha}{2_{s}^{\ast}}=\frac{1}{p}$. Then the conclusion follows.
\end{proof}

\begin{lemma}\label{high}
Let $K\in L^{\frac{6}{6s-3}}(\mathbb{R}^{3})$ with $s>\frac{1}{2}$, then $\Phi$ and $N$ have following properties:\\
(1) $\Phi: D^{s,2}(\mathbb{R}^{3})\mapsto D^{s,2}(\mathbb{R}^{3})$ is continuous and maps bounded sets into bounded sets;\\
(2) If $u_{n}\rightharpoonup u$ in $D^{s,2}(\mathbb{R}^{3})$, then up to a subsequence $\Phi(u_{n})\rightarrow \Phi(u)$ in $D^{s,2}(\mathbb{R}^{3})$ and $N(u_{n})\rightarrow N(u)$.
\end{lemma}
\begin{proof}
(1) By (\ref{dt3}), we see that
\begin{align}\label{seq}
\|T_{u}\|_{\mathcal{L}(D^{s,2},\mathbb{R})}=\|\Phi(u)\|_{D^{s,2}}, \ \forall \ u\in D^{s,2}(\mathbb{R}^{3}).
\end{align}
Assume that $u_{n}\in D^{s,2}(\mathbb{R}^{3})$, $u_{n}\rightarrow u$ in $D^{s,2}(\mathbb{R}^{3})$. Since $D^{s,2}(\mathbb{R}^{3})\hookrightarrow L^{\frac{6}{3-2s}}(\mathbb{R}^{3})$, we have $||u_{n}-u||\rightarrow 0$ in $L^{\frac{6}{3-2s}}(\mathbb{R}^{3})$, then $||u_{n}-u||^{\frac{3}{2s}}\rightarrow 0$ in $L^{\frac{4s}{3-2s}}(\mathbb{R}^{3})$.  Since $K\in L^{\frac{6}{6s-3}}(\mathbb{R}^{3})$, we have $\int_{\mathbb{R}^{3}}|K(x)|^{\frac{3}{2s}}|u_{n}-u|^{\frac{3}{2s}}dx\rightarrow 0$.
For any $v\in D^{s,2}(\mathbb{R}^{3})$, by H\"{o}lder inequality,
\begin{align*}
|T_{u_{n}}(v)-T_{u}(v)|&\leq \int_{\mathbb{R}^{3}}K(x)|(u^{2}_{n}-u^{2})v|dx\\
&\leq\left(\int_{\mathbb{R}^{3}}|K(x)|^{\frac{3}{2s}}|u_{n}-u|^{\frac{3}{2s}}dx\right)^{\frac{2s}{3}}\|u_{n}+u\|_{\frac{6}{3-2s}}\|v\|_{2^{\ast}_{s}}\\
&\rightarrow 0.
\end{align*}
Since $v$ is arbitrary, we get $\lim_{n\rightarrow \infty}\|T_{u_{n}}-T_{u}\|_{\mathcal{L}(D^{s,2},\mathbb{R})}=0$. By (\ref{ph}), $\Phi$ maps bounded sets into bounded sets. \\
(2) By (\ref{seq}), we only need to show that
\begin{align*}
\|T_{u_{n}}-T_{u}\|_{\mathcal{L}(D^{s,2},\mathbb{R})}\rightarrow 0 \ \mbox{as} \ n\rightarrow \infty.
\end{align*}
Let $\varepsilon$ be any given positive number. There exists $r_{\varepsilon}>0$ large enough such that $\|K\|_{L^{\frac{6}{6s-3}}(\mathbb{R}^{3}\backslash B(0,r_{\varepsilon}))}<\varepsilon$. For any $v\in D^{s,2}(\mathbb{R}^{3})$, by H\"{o}lder inequality,
\begin{align}\label{sequence}
|T_{u_{n}}(v)-T_{u}(v)|&=\int_{\mathbb{R}^{3}}K(x)|(u^{2}_{n}-u^{2})v|dx\nonumber\\
&\leq \int_{\mathbb{R}^{3}\backslash B(0,r_{\varepsilon})}K(x)|u^{2}_{n}-u^{2}||v|dx+\int_{B(0,r_{\varepsilon})}K(x)|u^{2}_{n}-u^{2}||v|dx\nonumber\\
&\leq\|K\|_{L^{\frac{6}{6s-3}}(\mathbb{R}^{3}\backslash B(0,r_{\varepsilon}))}\|u^{2}_{n}-u^{2}\|_{\frac{3}{3-2s}}\|v\|_{\frac{6}{3-2s}}\nonumber\\
&\quad+\left(\int_{B(0,r_{\varepsilon})}(K(x))^{\frac{6}{3+2s}}|u^{2}_{n}-u^{2}|^{\frac{6}{3+2s}}dx\right)^{\frac{3+2s}{6}}\|v\|_{2^{\ast}_{s}}\nonumber\\
&\leq \left(\varepsilon C+\left(\int_{B(0,r_{\varepsilon})}(K(x))^{\frac{6}{3+2s}}|u^{2}_{n}-u^{2}|^{\frac{6}{3+2s}}dx\right)^{\frac{3+2s}{6}}\|v\|_{2^{\ast}_{s}}\right)\|v\|_{D^{s,2}}.
\end{align}
Set $B_{M}=\{x\in B(0,r_{\varepsilon}):K(x)>M\}$. Since $K\in L^{\frac{6}{6s-3}}(\mathbb{R}^{3})$, then the volume $|B_{M}|\rightarrow 0$ as $M\rightarrow \infty$. By Lemma \ref{Hardy}, up to a subsequence,
$u_{n}\rightarrow u$ in $L_{loc}^{\frac{12}{3+2s}}(\mathbb{R}^{3})$. Then for large $M$, we have
\begin{align}\label{largeM}
&\left(\int_{B(0,r_{\varepsilon})}(K(x))^{\frac{6}{3+2s}}|u^{2}_{n}-u^{2}|^{\frac{6}{3+2s}}dx\right)^{\frac{3+2s}{6}}\nonumber\\
&=\int_{B_{M}}(K(x))^{\frac{6}{3+2s}}|u_{n}+u|^{\frac{6}{3+2s}}|u_{n}-u|^{\frac{6}{3+2s}}dx\nonumber\\
&\quad+\int_{B(0,r_{\varepsilon})\backslash B_{M}}(K(x))^{\frac{6}{3+2s}}|u_{n}+u|^{\frac{6}{3+2s}}|u_{n}-u|^{\frac{6}{3+2s}}dx\nonumber\\
&\leq \left(\int_{B_{M}}(K(x))^{\frac{6}{6s-3}}dx\right)^{\frac{6s-3}{3+2s}}\left(\int_{\mathbb{R}^{3}}|u_{n}+u|^{\frac{6}{3-2s}}dx\right)^{\frac{3-2s}{3+2s}}\left(\int_{\mathbb{R}^{3}}|u_{n}-u|^{\frac{6}{3-2s}}dx\right)^{\frac{3-2s}{3+2s}}\nonumber\\
&\quad+M^{\frac{6}{3+2s}}\left(\int_{B(0,r_{\varepsilon})}|u_{n}+u|^{\frac{12}{3+2s}}dx\right)^{\frac{1}{2}}\left(\int_{B(0,r_{\varepsilon})}|u_{n}-u|^{\frac{12}{3+2s}}dx\right)^{\frac{1}{2}}\nonumber\\
&\leq \varepsilon c+co(1).
\end{align}
Therefore,
\begin{align*}
\Phi(u_{n})\rightarrow \Phi(u) \ \mbox{in} \ D^{s,2}(\mathbb{R}^{3}).
\end{align*}
Since $D^{s,2}(\mathbb{R}^{3})\hookrightarrow L^{2_{s}^{\ast}}(\mathbb{R}^{3})$ and $\phi_{u_{n}}^{s}\rightarrow \phi_{u}^{s}$ in $D^{s,2}(\mathbb{R}^{3})$, then by H\"{o}lder inequality, we have
\begin{align}\label{embedding}
\int_{\mathbb{R}^{3}}K(x)(\phi_{u_{n}}^{s}-\phi_{u}^{s})u^{2}dx\leq \|K\|_{\frac{6}{6s-3}}\|u\|^{2}_{2_{s}^{\ast}}\|\phi_{u_{n}}^{s}-\phi_{u}^{s}\|_{2_{s}^{\ast}}=o(1).
\end{align}
By the same argument as in (\ref{sequence}) and (\ref{largeM}), replacing $v$ by $\phi_{u_{n}}^{s}$,  we get
\begin{align}\label{N}
\int_{\mathbb{R}^{3}}K(x)\phi_{u_{n}}^{s}(u_{n}^{2}-u^{2})dx \rightarrow 0 \ \mbox{as} \ n\rightarrow \infty.
\end{align}
From (\ref{embedding}), (\ref{N}) and
\begin{align*}
N(u_{n})-N(u)=\int_{\mathbb{R}^{3}}K(x)(\phi_{u_{n}}^{s}-\phi_{u}^{s})u^{2}dx+\int_{\mathbb{R}^{3}}K(x)\phi_{u_{n}}^{s}(u_{n}^{2}-u^{2})dx,
\end{align*}
it follows that $N(u_{n})\rightarrow N(u)$ as $n\rightarrow \infty$.
\end{proof}
We consider the limiting problem
\begin{equation}\label{linear}
 \ \left\{\begin{aligned}
&(-\Delta)^{s}u=|u|^{2^{\ast}_{s}-2}u \ \ \ \ \mbox{in} \ \mathbb{R}^{3},\\
&u\in D^{s,2}(\mathbb{R}^{3}).
\end{aligned}\right.
\end{equation}
It is known that every positive solution of (\ref{linear}) assumes the form (see \cite {Chen})
\begin{equation}\label{form}
\Psi_{\delta,x_{0},\kappa}(x)=\kappa\left(\frac{ \delta}{\delta^{2}+|x-x_{0}|^{2}}\right)^{\frac{3-2s}{2}}, \ \kappa>0,\ \delta>0, \ x_{0}\in \mathbb{R}^{3}.
\end{equation}
By \cite{C},  the  best constant $S_{s}$ is only attained at $\Psi_{\delta,x_{0},\kappa}(x)$. It is known that \cite {Cora} if we choose
\begin{equation*}
\kappa=b_{s}=2^{\frac{3-2s}{2}}\left(\frac{\Gamma(\frac{3+2s}{2})}{\Gamma(\frac{3-2s}{2})}\right)^{\frac{3-2s}{4s}},
\end{equation*}
then
\begin{equation}\label{psi}
\psi_{\delta,x_{0}}(x)=b_{s}\left(\frac{ \delta}{\delta^{2}+|x-x_{0}|^{2}}\right)^{\frac{3-2s}{2}}
\end{equation}
satisfies (\ref{linear}), and
\begin{equation}\label{energy}
\int_{\mathbb{R}^{3}}|(-\Delta)^{\frac{s}{2}}\psi_{\delta,x_{0}}|^{2}dx=\int_{\mathbb{R}^{3}}|\psi_{\delta,x_{0}}|^{2_{s}^{\ast}}dx=S_{s}^{\frac{3}{2s}}.
\end{equation}
The energy functional related to (\ref{linear}) is defined on $D^{s,2}(\mathbb{R}^{3})$ by
\begin{equation*}
I_{\infty}(u)=\frac{1}{2}\int_{\mathbb{R}^{3}}|(-\Delta)^{\frac{s}{2}}u|^{2}dx-\frac{1}{2_{s}^{\ast}}\int_{\mathbb{R}^{3}}|u|^{2_{s}^{\ast}}dx.
\end{equation*}
Then
\begin{equation*}
I_{\infty}(\psi_{\delta,x_{0}})=\frac{s}{3}S_{s}^{\frac{3}{2s}}.
\end{equation*}

We define the Nehari manifold corresponding to (\ref{fsp}) and (\ref{linear}) by $\mathcal{N}$ and $\mathcal{N}_{\infty}$:
\begin{align}\label{Nehari}
\mathcal{N}&=\{u\in D^{s,2}(\mathbb{R}^{3})\backslash \{0\}: \langle I'(u), u\rangle=0\}, \nonumber\\
\mathcal{N}_{\infty}&=\{u\in D^{s,2}(\mathbb{R}^{3})\backslash \{0\}: \langle I_{\infty}'(u), u\rangle=0\}.
\end{align}
Set
\begin{align}\label{inf}
m=\inf_{u\in \mathcal{N}} I(u), \ m_{\infty}=\inf_{u\in \mathcal{N}_{\infty}} I_{\infty}(u).
\end{align}
\begin{lemma}\label{minimization}
Assume that $K\in L^{\frac{6}{6s-3}}(\mathbb{R}^{3})$ and $V\in L^{\frac{3}{2s}}(\mathbb{R}^{3})$ are nonnegative functions. Then $m=m_{\infty}$ and $m$ can't be attained when $\|V\|_{\frac{3}{2s}}+\|K\|_{\frac{6}{6s-3}}>0$.
\end{lemma}

\begin{proof}
Let $u\in D^{s,2}(\mathbb{R}^{3})$ and define the function
\begin{align}
\gamma(t):&=\langle I'(tu),tu\rangle\nonumber\\
&=t^{2}\int_{\mathbb{R}^{3}}(|(-\Delta)^{\frac{s}{2}}u|^{2}+V(x)u^{2})dx+t^{4}\int_{\mathbb{R}^{3}}K(x)\phi_{u}^{s}u^{2}dx-t^{2_{s}^{\ast}}\int_{\mathbb{R}^{3}}|u|^{2_{s}^{\ast}}dx\nonumber\\
&=at^{2}+bt^{4}-ct^{2_{s}^{\ast}},
\end{align}
where $a=\int_{\mathbb{R}^{3}}(|(-\Delta)^{\frac{s}{2}}u|^{2}+V(x)u^{2})dx$, $b=\int_{\mathbb{R}^{3}}K(x)\phi_{u}^{s}u^{2}dx$, $c=\int_{\mathbb{R}^{3}}|u|^{2_{s}^{\ast}}dx$.
Since $s>\frac{3}{4}$, by a similar argument as the proof of Lemma 3.3 in \cite{Ruiz}, we know that there exist unique $t(u)>0$, $s(u)>0$ such that $t(u)u\in \mathcal{N},s(u)u\in \mathcal{N}_{\infty}$, and $I(t(u)u)=\max_{t>0}I(tu)$, $I_{\infty}(s(u)u)=\max_{t>0}I_{\infty}(tu)$. For any $u\in \mathcal{N}$,
\begin{align}\label{inequality}
m_{\infty}&\leq I_{\infty}(s(u)u)\nonumber\\
&\leq\frac{1}{2}\|s(u)u\|_{D^{s,2}}^{2}+\frac{s^{2}(u)}{2}\int_{\mathbb{R}^{3}}V(x)u^{2}dx+\frac{1}{4}N(s(u)u)-\frac{1}{2_{s}^{\ast}}\|s(u)u\|_{2_{s}^{\ast}}^{2_{s}^{\ast}}\nonumber\\
&=I(s(u)u).
\end{align}
Since $u\in \mathcal{N}$, it is easy to see that $I(tu)$ attains its maximum at $t=1$, then for any $t>0$,
\begin{align}\label{attain}
I(tu)\leq I(u).
\end{align}
By (\ref{inequality}), (\ref{attain}), we have
\begin{align}\label{first}
m_{\infty}\leq m.
\end{align}
Assume that $w$ is a positive solution of (\ref{linear}) centered at zero, $(z_{n})_{n}\subset \mathbb{R}^{3}$ satisfying $|z_{n}|\rightarrow \infty$ as $n\rightarrow \infty$,
$w_{n}(\cdot)=w(\cdot-z_{n})$ and $t_{n}=t(w_{n})$. It is clear that $\|w_{n}\|_{D^{s,2}}=\|w\|_{D^{s,2}}$, $\|w_{n}\|_{2_{s}^{\ast}}=\|w\|_{2_{s}^{\ast}}$ and $w_{n}\rightharpoonup 0$ in $D^{s,2}(\mathbb{R}^{3})$ as $n\rightarrow \infty$. It follows from Lemma \ref{high} that
\begin{align}\label{nw}
N(w_{n})\rightarrow 0.
\end{align}
Since $D^{s,2}(\mathbb{R}^{3})\hookrightarrow L^{2_{s}^{\ast}}(\mathbb{R}^{3})$, then $w_{n}^{2}\rightharpoonup 0$ in $L^{\frac{2_{s}^{\ast}}{2}}(\mathbb{R}^{3})$, this together with $V\in L^{\frac{3}{2s}}(\mathbb{R}^{3})$ yield that
\begin{align}\label{vw}
\int_{\mathbb{R}^{3}}V(x)w_{n}^{2}\rightarrow 0.
\end{align}
By (\ref{nw}), (\ref{vw}), we have
\begin{align}\label{iun}
I(u_{n})=\frac{t_{n}^{2}}{2}\|w\|_{D^{s,2}}+\frac{t_{n}^{2}}{2}o_{n}(1)+\frac{t_{n}^{4}}{4}o_{n}(1)-\frac{t_{n}^{2_{s}^{\ast}}}{2_{s}^{\ast}}\|w\|_{2_{s}^{\ast}}^{2_{s}^{\ast}}.
\end{align}
Since $w_{n}\in \mathcal{N}_{\infty}$ and $t_{n}w_{n}\in \mathcal{N}$, then
\begin{align}\label{w}
\|w_{n}\|_{D^{s,2}}^{2}=\|w_{n}\|_{2_{s}^{\ast}}^{2_{s}^{\ast}},
\end{align}
and
\begin{align}\label{tn}
t_{n}^{2}\|w_{n}\|_{D^{s,2}}^{2}+t_{n}^{2}\int_{\mathbb{R}^{3}}V(x)w_{n}^{2}dx+t_{n}^{4}N(w_{n})=t_{n}^{2_{s}^{\ast}}\|w_{n}\|_{2_{s}^{\ast}}^{2_{s}^{\ast}}.
\end{align}
From (\ref{w}) and (\ref{tn}), it follows that
\begin{align}\label{imply}
(1-t_{n}^{2_{s}^{\ast}-2})\|w_{n}\|_{D^{s,2}}^{2}+\int_{\mathbb{R}^{3}}V(x)w_{n}^{2}dx+t_{n}^{2}N(w_{n})=0.
\end{align}
By (\ref{nw}), (\ref{vw}) and (\ref{tn}), for sufficiently large $n$, we have
\begin{align*}
\|w\|_{2_{s}^{\ast}}^{2_{s}^{\ast}}t_{n}^{2_{s}^{\ast}-2}\leq \|w\|_{D^{s,2}}^{2}+o_{n}(1)+t_{n}^{2}o_{n}(1).
\end{align*}
Then $(t_{n})_{n}$ is bounded. This together with $\|w_{n}\|_{D^{s,2}}=\|w\|_{D^{s,2}}$ and (\ref{imply}) yield that $\lim_{n\rightarrow \infty}t_{n}=1$.
By (\ref{iun}), we see that  $\lim_{n\rightarrow \infty}I(u_{n})=m_{\infty}$. Therefore by (\ref{inf}) and (\ref{first}), $m=m_{\infty}$.

Now we will prove that $m$ can't be attained. Arguing by contradiction, we let $\bar{u}\in \mathcal{N}$ be a function such that $I(\bar{u})=m=m_{\infty}$. Then, since $s(\bar{u})\bar{u}\in \mathcal{N}_{\infty}$, we have
\begin{align*}
m_{\infty}&\leq I_{\infty}(s(\bar{u})\bar{u})\\
&\leq \frac{1}{2}\|s(\bar{u})\bar{u}\|_{D^{s,2}}^{2}+\frac{s^{2}(\bar{u})}{2}\int_{\mathbb{R}^{3}}V(x)\bar{u}^{2}dx+\frac{1}{4}N(s(\bar{u})\bar{u})-\frac{1}{2_{s}^{\ast}}\|s(\bar{u})\bar{u}\|_{2_{s}^{\ast}}^{2_{s}^{\ast}}\\
&\leq I(\bar{u})=m_{\infty}.
\end{align*}
We deduce that
\begin{align*}
\frac{s^{2}(\bar{u})}{2}\int_{\mathbb{R}^{3}}V(x)\bar{u}^{2}dx=0, \ N(s(\bar{u})\bar{u})=0, \ s(\bar{u})=1,
\end{align*}
then $\bar{u}\in \mathcal{N}_{\infty}$, $I_{\infty}(\bar{u})=m_{\infty}$. Therefore $\bar{u}$ admits the form (\ref{form}), contradicting $\int_{\mathbb{R}^{3}}V(x)\bar{u}^{2}dx=0$ and $N(\bar{u})=0$.
\end{proof}

\section{\textbf{A compactness theorem}}
\begin{lemma}\label{yang}(see Lemma 2.5 in \cite{Yang})
Let $\{u_{n}\}\subset H_{loc}^{s}(\mathbb{R}^{N})$ be a bounded sequence of functions such that $u_{n}\rightharpoonup 0$ in $H^{s}(\mathbb{R}^{N})$. Suppose that there exist a bounded open set $Q\subset \mathbb{R}^{N}$ and a
positive constant $\gamma>0$ such that
\begin{align*}
\int_{Q}|(-\triangle)^{s}u_{n}|dx\geq \gamma>0, \ \int_{Q}|u_{n}|^{2_{s}^{\ast}}dx\geq \gamma>0.
\end{align*}
Moreover suppose that
\begin{align*}
(-\Delta)^{s}u_{n}-|u_{n}|^{2_{s}^{\ast}-2}u_{n}=\chi_{n},
\end{align*}
where $\chi_{n}\in H^{-s}(\mathbb{R}^{N})$ and
\begin{align*}
|\langle\chi_{n},\varphi\rangle|\leq \varepsilon_{n}\|\varphi\|_{H^{s}(\mathbb{R}^{N})}, \ \forall \ \varphi\in C_{0}^{\infty}(U),
\end{align*}
where $U$ is an open neighborhood of $Q$ and $\varepsilon_{n}$ is a sequence of positive numbers converging to 0. Then there exist a sequence of points $\{y_{n}\}\subset \mathbb{R}^{N}$ and a sequence of positive numbers $\{\sigma_{n}\}$ such that
\begin{align*}
v_{n}:=\sigma_{n}^{\frac{N-2s}{2}}u_{n}(\sigma_{n}x+y_{n})
\end{align*}
converges weakly in $D^{s,2}(\mathbb{R}^{N})$ to a nontrivial solution of
\begin{align*}
(-\Delta)^{s}u=|u|^{2_{s}^{\ast}-2}, \ u\in D^{s,2}(\mathbb{R}^{N}),
\end{align*}
morover,
\begin{align*}
y_{n}\rightarrow \bar{y}\in \bar{Q} \ \mbox{and} \ \sigma_{n}\rightarrow 0.
\end{align*}
\end{lemma}
\begin{lemma}\label{unps}
Let $\{u_{n}\}$ be a Palais-Smale sequence for $I$, such that $u_{n}\in C_{0}^{\infty}(\mathbb{R}^{3})$ and
\begin{align*}
&u_{n}\rightharpoonup 0 \ weakly \ in \ D^{s,2}(\mathbb{R}^{3})\\
&u_{n}\nrightarrow 0 \ strongly \ in \ D^{s,2}(\mathbb{R}^{3}).
\end{align*}
Then there exist a sequence of points $\{y_{n}\}\subset \mathbb{R}^{3}$ and a sequence of positive numbers $\{\sigma_{n}\}$ such that
\begin{align*}
v_{n}(x)=\sigma_{n}^{\frac{3-2s}{2}}u_{n}(\sigma_{n}x+y_{n})
\end{align*}
converges weakly in $D^{s,2}(\mathbb{R}^{3})$ to a nontrival solution of (\ref{linear}) and
\begin{align}\label{iun1}
I(u_{n})=I_{\infty}(v)+I_{\infty}(v_{n}-v)+o(1)
\end{align}
\begin{align}\label{iun2}
\|u_{n}\|_{D^{s,2}}^{2}&=\|v\|_{D^{s,2}}^{2}+\|v_{n}-v\|_{D^{s,2}}^{2}+o(1).
\end{align}
\end{lemma}
\begin{proof}
Since $u_{n}\rightharpoonup 0 \ in \ D^{s,2}(\mathbb{R}^{3})$, by (\ref{vw}) and (2) of Lemma \ref{high}, we get
\begin{align}\label{get}
\int_{\mathbb{R}^{3}}V(x)u_{n}^{2}\rightarrow 0, \  \ N(u_{n})\rightarrow 0.
\end{align}
For any $h\in D^{s,2}(\mathbb{R}^{3})$, by the H\"{o}lder inequality,
\begin{align*}
\int_{\mathbb{R}^{3}}V(x)u_{n}hdx\leq \left(\int_{\mathbb{R}^{3}}(V(x)u_{n})^{\frac{6}{3+2s}} dx\right)^{\frac{3+2s}{6}}\left(\int_{\mathbb{R}^{3}}|h|^{2_{s}^{\ast}}dx\right)^{\frac{1}{2_{s}^{\ast}}}.
\end{align*}
From $u_{n}\rightharpoonup 0$ in $L^{2_{s}^{\ast}}(\mathbb{R}^{3})$, we have $u_{n}^{\frac{6}{3+2s}}\rightharpoonup 0$ in $L^{\frac{3+2s}{3-2s}}(\mathbb{R}^{3})$.
This together with $V^{\frac{6}{3+2s}} \in L^{\frac{3+2s}{4s}}(\mathbb{R}^{3})$ yield that
\begin{align*}
\int_{\mathbb{R}^{3}}V^{\frac{6}{3+2s}}u_{n}^{\frac{6}{3+2s}}dx\rightarrow 0.
\end{align*}
Since $D^{s,2}(\mathbb{R}^{3})\hookrightarrow L^{2_{s}^{\ast}}(\mathbb{R}^{3})$, then $\|h\|_{2_{s}^{\ast}}\leq C\|h\|_{D^{s,2}}$. So
\begin{align}\label{weak}
\int_{\mathbb{R}^{3}}V(x)u_{n}hdx=o(1)\|h\|_{D^{s,2}}.
\end{align}
By the similar argument as that in the proof of (2) in Lemma \ref{high}, we have
\begin{align}\label{strong}
\int_{\mathbb{R}^{3}}K(x)\phi_{u_{n}}^{s}u_{n}hdx=o(1)\|h\|_{D^{s,2}}.
\end{align}
Therefore
\begin{align}\label{therefore1}
I(u_{n})=I_{\infty}(u_{n})+o(1),
\end{align}
\begin{align}\label{therefore2}
I'(u_{n})&=I'_{\infty}(u_{n})+o(1)=o(1).
\end{align}
Put $\sigma_{n}=\|u_{n}\|_{2}^{1/s}$ and define
\begin{align*}
\tilde{u}_{n}(x)=\sigma_{n}^{\frac{3-2s}{2}}u_{n}(\sigma_{n}x).
\end{align*}
From Proposition 3.6 in \cite{DPV}, up to a positive constant, it follows that
\begin{align}\label{equivalent}
\|(-\Delta)^{\frac{s}{2}}u\|_{2}^{2}=\int_{\mathbb{R}^{3}}\int_{\mathbb{R}^{3}}\frac{|u(x)-u(y)|^{2}}{|x-y|^{3+2s}}dxdy.
\end{align}
Then we can get
\begin{align}\label{two norm}
\|\tilde{u}_{n}\|_{D^{s,2}}=\|u_{n}\|_{D^{s,2}},  \ \|\tilde{u}_{n}\|_{2_{s}^{\ast}}=\|u_{n}\|_{2_{s}^{\ast}}.
\end{align}
By the invariance of these two norms with respect to rescaling $u\rightarrow r^{\frac{3-2s}{2}}u(r\cdot)$, note that (\ref{therefore2}), we have
\begin{align}\label{recaling}
I_{\infty}(\tilde{u}_{n})=I_{\infty}(u_{n}), \ I'_{\infty}(\tilde{u}_{n})=I'_{\infty}(u_{n})=o(1).
\end{align}
By the scale change,
\begin{align}\label{one}
\|\tilde{u}_{n}\|_{2}=1.
\end{align}
If $\tilde{u}_{n}\rightharpoonup \tilde{u}\neq 0$ in $D^{s,2}(\mathbb{R}^{3})$, we are done: the wanted functions $v_{n}$ and $v$ are precisely $\tilde{u}_{n}$ and $\tilde{u}$, and $y_{n}=0$. In fact, since $\tilde{u}_{n}\rightharpoonup \tilde{u}$ in $D^{s,2}(\mathbb{R}^{3})$, by the Br$\acute{\mbox{e}}$zis-Lieb Lemma \cite{Willem},
\begin{align}\label{Brezis}
I_{\infty}(\tilde{u}_{n}-\tilde{u})=I_{\infty}(\tilde{u}_{n})-I_{\infty}(\tilde{u})+o(1).
\end{align}
From (\ref{therefore1}), (\ref{therefore2}), (\ref{recaling}) and (\ref{Brezis}), it follows that
\begin{align}\label{that}
I(u_{n})=I_{\infty}(\tilde{u})+I_{\infty}(\tilde{u}_{n}-\tilde{u})+o(1).
\end{align}
By (\ref{two norm}),
\begin{align*}
\|u_{n}\|_{D^{s,2}}^{2}=\|\tilde{u}_{n}\|_{D^{s,2}}^{2}=\|\tilde{u}\|_{D^{s,2}}^{2}+\|\tilde{u}_{n}-\tilde{u}\|_{D^{s,2}}^{2}+o(1).
\end{align*}

If $\tilde{u}_{n}\rightharpoonup 0$ in $D^{s,2}(\mathbb{R}^{3})$. We cover $\mathbb{R}^{3}$ with balls of radius 1 in such a way that each point of $\mathbb{R}^{3}$ is contained in at most 4 balls. Set
\begin{align*}
d_{n}=\sup_{B_{1}}|\tilde{u}_{n}|_{L^{2_{s}^{\ast}}(B_{1})}=\sup_{x\in \mathbb{R}^{3}}\left(\int_{B_{1}(x)}|\tilde{u}_{n}|^{2_{s}^{\ast}}\right)^{\frac{1}{2_{s}^{\ast}}}.
\end{align*}
We claim that there exists $\gamma>0$ such that
\begin{align}\label{dn}
d_{n}\geq \gamma>0, \ \forall \ n\in \mathbb{N}.
\end{align}
Otherwise, by Lemma 2.2 in \cite{DSS}, we have
\begin{align}\label{ast}
\tilde{u}_{n}\rightarrow 0 \ \ \mbox{in} \ \ L^{2_{s}^{\ast}}(\mathbb{R}^{3}).
\end{align}
By (\ref{therefore2}), $\langle I'_{\infty}(\tilde{u}_{n}), \tilde{u}_{n}\rangle=o(1)$. This together with (\ref{ast}) yield that
\begin{align*}
\|\tilde{u}_{n}\|_{D^{s,2}}^{2}=\|\tilde{u}_{n}\|_{2_{s}^{\ast}}^{2_{s}^{\ast}}+o(1)=o(1).
\end{align*}
Then by (\ref{two norm}),  we have $\|u_{n}\|_{D^{s,2}}=o(1)$. This contradicts $u_{n}\nrightarrow 0$  strongly in $D^{s,2}(\mathbb{R}^{3})$.

We choose $\tilde{y}_{n}$ such that
\begin{align}\label{yn}
\int_{B_{1}(\tilde{y}_{n})}|\tilde{u}_{n}|^{2_{s}^{\ast}}dx\geq d_{n}-\frac{1}{n}
\end{align}
and let
\begin{align}\label{wn}
w_{n}=\tilde{u}_{n}(x+\tilde{y}_{n}).
\end{align}
By (\ref{dn}), (\ref{yn}) and (\ref{wn}),
\begin{align}\label{wngamma}
\int_{B_{1}(0)}|w_{n}|^{2_{s}^{\ast}}dx\geq \frac{\gamma}{2}.
\end{align}
If $w_{n}\rightharpoonup w\neq 0$ in $D^{s,2}(\mathbb{R}^{3})$, we are done: the wanted functions $v_{n}$ and $v$ are precisely $w_{n}$ and $w$, and $y_{n}=\sigma_{n}\tilde{y}_{n}$.
Suppose that $w_{n}\rightharpoonup 0$ in $D^{s,2}(\mathbb{R}^{3})$. By Lemma \ref{Hardy}, $w_{n}\rightarrow 0$ in $L_{loc}^{2}(\mathbb{R}^{3})$. Then by (\ref{wngamma}), we have
\begin{align*}
\int_{B_{1}(0)}|(-\Delta)^{\frac{s}{2}}w_{n}|^{2}dx&=\|w_{n}\|_{H^{s}(B_{1}(0))}^{2}-\|w_{n}\|_{L^{2}(B_{1}(0))}^{2}\geq C\|w_{n}\|_{L^{2_{s}^{\ast}}(B_{1}(0))}^{2}+o(1)\\
&\geq C\left(\frac{\gamma}{2}\right)^{\frac{2}{2_{s}^{\ast}}}+o(1)>0.
\end{align*}
By (\ref{recaling}), (\ref{wn}), the invariance of the $H^{s}(\mathbb{R}^{3})$ and $L^{2_{s}^{\ast}}(\mathbb{R}^{3})$ norms by translation, we apply Lemma \ref{yang} to $w_{n}$ and to claim that
there exists a sequence $\eta_{n}\rightarrow 0$ of numbers and a sequence $y_{n}$ of points: $y_{n}\rightarrow \bar{y}\in \overline{B_{1}(0)}$ such that
\begin{align*}
v_{n}=\eta_{n}^{\frac{3-2s}{2}}w_{n}(\eta_{n}x+y_{n})\rightharpoonup v\neq 0.
\end{align*}
By (\ref{two norm}), it is easy to see that
\begin{align}\label{vu}
\|v_{n}\|_{D^{s,2}}=\|u_{n}\|_{D^{s,2}}, \ \|v_{n}\|_{2_{s}^{\ast}}=\|u_{n}\|_{2_{s}^{\ast}}.
\end{align}
Since $v_{n}\rightharpoonup v$ in $D^{s,2}(\mathbb{R}^{3})$, by the Br$\acute{\mbox{e}}$zis-Lieb Lemma \cite{Willem}, (\ref{therefore1}), (\ref{therefore2}) and (\ref{vu}), we have
\begin{align*}
I(u_{n})&=I_{\infty}(u_{n})+o(1)=I_{\infty}(v_{n})+o(1)=I_{\infty}(v)+I_{\infty}(v_{n}-v)+o(1),\\
\|u_{n}\|_{D^{s,2}}^{2}&=\|v_{n}\|_{D^{s,2}}^{2}=\|v\|_{D^{s,2}}^{2}+\|v_{n}-v\|_{D^{s,2}}^{2}+o(1).
\end{align*}
So $v_{n}$ and $v$ are the wanted functions.
\end{proof}
\begin{theorem}\label{compactness}
Suppose that $V\in L^{\frac{3}{2s}}(\mathbb{R}^{3})$, $V\geq 0$ and $K\in L^{\frac{6}{6s-3}}(\mathbb{R}^{3})$, $K\geq 0$.
Let $\{u_{n}\}$ be a Palsis-Smale sequence of $I$ at level $c$. Then up to a subsequence, $u_{n}\rightharpoonup \bar{u}$. Moreover, there exists a number $k\in \mathbb{N}$, $k$ sequences of points $\{y_{n}^{i}\}_{n}\subset \mathbb{R}^{3}$, $1\leq i\leq k$, $k$ sequences of positive numbers $\{\sigma_{n}^{i}\}_{n}$, $1\leq i\leq k$, $k$ sequences of functions $\{u_{n}^{i}\}\subset D^{s,2}(\mathbb{R}^{3})$, $1\leq i\leq k$, such that
\begin{align*}
u_{n}(x)-\bar{u}(x)-\sum_{i=1}^{k}\frac{1}{(\sigma_{n}^{i})^{\frac{3-2s}{2}}}u^{i}\left(\frac{x-y_{n}^{i}}{\sigma_{n}^{i}}\right)\rightarrow 0,
\end{align*}
where $u^{i}$, $1\leq i \leq k$, are solutions of (\ref{linear}). Moreover as $n\rightarrow \infty$,
\begin{align*}
\|u_{n}\|_{D^{s,2}}^{2}\rightarrow \|\bar{u}\|_{D^{s,2}}^{2}+\sum_{i=1}^{k}\|u^{i}\|_{D^{s,2}}^{2}
\end{align*}
\begin{align*}
I(u_{n})\rightarrow I(\bar{u})+\sum_{i=1}^{k}I_{\infty}(u^{i}).
\end{align*}
\end{theorem}
\begin{proof}
Since $u_{n}$ is a Palais-Smale sequence for $I$, it is easy to show that $u_{n}$ is bounded in $D^{s,2}(\mathbb{R}^{3})$ and in $L^{2_{s}^{\ast}}(\mathbb{R}^{3})$. Up to a subsequence, we assume that $u_{n}\rightharpoonup \bar{u}$ in $D^{s,2}(\mathbb{R}^{3})$ and in $L^{2_{s}^{\ast}}(\mathbb{R}^{3})$ as $n\rightarrow \infty$. By the weak convergence, we have $I'(\bar{u})=0$. Note that $V\in L^{\frac{3}{2s}}(\mathbb{R}^{3})$, $V\geq 0$ and $K\in L^{\frac{6}{6s-3}}(\mathbb{R}^{3})$, $K\geq 0$, by Proposition 5.1.1 in \cite{Dipierro} and Lemma 2.7 in \cite{Teng}, we have $\bar{u}\in L^{\infty}(\mathbb{R}^{3})$. Set $z_{n}^{1}=u_{n}-\bar{u}$, then $z_{n}^{1}\rightharpoonup 0$ in $D^{s,2}(\mathbb{R}^{3})$. If $z_{n}^{1}\rightarrow 0$ in $D^{s,2}(\mathbb{R}^{3})$, we are done. Suppose that $z_{n}^{1}=u_{n}-\bar{u}\nrightarrow 0$ in $D^{s,2}(\mathbb{R}^{3})$, for any $v\in C_{0}^{\infty}(\mathbb{R}^{3})$,
\begin{align}\label{zh}
\langle I'(z_{n}^{1}),v\rangle&=\int_{\mathbb{R}^{3}}(-\Delta)^{\frac{s}{2}}z_{n}^{1}(-\Delta)^{\frac{s}{2}}vdx+\int_{\mathbb{R}^{3}}V(x)z_{n}^{1}vdx
+\int_{\mathbb{R}^{3}}K(x)\phi_{z_{n}^{1}}^{s}z_{n}^{1}vdx\nonumber\\
&\quad-\int_{\mathbb{R}^{3}}|z_{n}^{1}|^{2_{s}^{\ast}-2}z_{n}^{1}vdx\nonumber\\
&=\langle I'(u_{n}), v\rangle-\langle I'(u), v\rangle+\int_{\mathbb{R}^{3}}K(x)(\phi_{z_{n}^{1}}^{s}z_{n}^{1}+\phi_{\bar{u}}^{s}\bar{u}-\phi_{u_{n}^{s}}^{s}u_{n})vdx\nonumber\\
&\quad-\int_{\mathbb{R}^{3}}(|z_{n}^{1}|^{2_{s}^{\ast}-2}z_{n}^{1}+|\bar{u}|^{2_{s}^{\ast}-2}\bar{u}-|u_{n}|^{2_{s}^{\ast}-2}u_{n})vdx.
\end{align}
Since $\bar{u}\in L^{\infty}(\mathbb{R}^{3})$, by arguments similar to those of Lemma 8.9 in \cite{Willem}, we get
\begin{align}\label{invariance}
|u_{n}-\bar{u}|^{2_{s}^{\ast}-2}(u_{n}-u)+|\bar{u}|^{2_{s}^{\ast}-2}\bar{u}-|u_{n}|^{2_{s}^{\ast}-2}u_{n}\rightarrow 0 \ \mbox{in}\ (D^{s,2}(\mathbb{R}^{3}))'.
\end{align}
Since $u_{n}\rightharpoonup \bar{u}$ in $D^{s,2}(\mathbb{R}^{3})$, it is easy to see that
\begin{align*}
&\int_{\mathbb{R}^{3}}K(x)\phi_{z_{n}^{1}}^{s}z_{n}^{1}vdx=o(1)\|v\|_{D^{s,2}},\\
&\int_{\mathbb{R}^{3}}K(x)\phi_{\bar{u}}^{s}\bar{u}-\phi_{u_{n}^{s}}^{s}u_{n})vdx=o(1)\|v\|_{D^{s,2}}.
\end{align*}
Then
\begin{align}\label{k}
\int_{\mathbb{R}^{3}}K(x)(\phi_{z_{n}^{1}}^{s}z_{n}^{1}vdx+\phi_{\bar{u}}^{s}\bar{u}-\phi_{u_{n}^{s}}^{s}u_{n})vdx=o(1)\|v\|_{D^{s,2}}.
\end{align}
From (\ref{zh}), (\ref{invariance}) and (\ref{k}), it follows that
$\langle I'(z_{n}^{1}),v\rangle=o(1)v$, then $z_{n}^{1}$ is a Palais-Smale sequence for $I$.\\
By the definition of $D^{s,2}(\mathbb{R}^{3})$, for any $n$, there exists $\tilde{z}_{n}^{1}\in C_{0}^{\infty}(\mathbb{R}^{3})$ such that
\begin{align*}
\|z_{n}^{1}-\tilde{z}_{n}^{1}\|_{D^{s,2}}< \frac{1}{n} \ \ \mbox{and} \ \ \|I'(z_{n}^{1})-I'(\tilde{z}_{n}^{1})\|_{D^{s,2}}< \frac{1}{n}.
\end{align*}
Thus  $\bar{z}_{n}^{1}$ is a Palais-Smale sequence, and
\begin{align}\label{strongly}
z_{n}^{1}-\tilde{z}_{n}^{1}\rightarrow 0 \ \ \mbox{in} \ D^{s,2}(\mathbb{R}^{3}),
\end{align}
\begin{align}\label{weakly}
\tilde{z}_{n}^{1}\rightharpoonup 0  \ \ \mbox{in} \ D^{s,2}(\mathbb{R}^{3}).
\end{align}
By a similar argument to that in proof of Lemma \ref{high}, we get that as $n\rightarrow \infty$,
\begin{align}\label{verify}
\left\{\begin{aligned}
&\int_{\mathbb{R}^{3}}K(x)\tilde{z}_{n}^{1}(x)(z_{n}^{1}-\tilde{z}_{n}^{1})(x)dx\rightarrow 0\\
&K(x)\bar{u}(x)\tilde{z}_{n}^{1}(x)dx\rightarrow 0\\
&K(x)\bar{u}(x)(z_{n}^{1}-\tilde{z}_{n}^{1})(x)dx\rightarrow 0.
\end{aligned}\right.
\end{align}
By (\ref{strongly}), (\ref{weakly}), (\ref{verify}) and the Br$\acute{\mbox{e}}$zis-Lieb Lemma, we conclude that
\begin{align}\label{conclude1}
\|\tilde{z}_{n}^{1}\|_{D^{s,2}}^{2}&=\|u_{n}\|_{D^{s,2}}^{2}-\|\tilde{u}\|_{D^{s,2}}^{2}+o(1),
\end{align}
\begin{align}\label{conclude2}
I(\tilde{z}_{n}^{1})&=I(u_{n})-I(\bar{u})+o(1),
\end{align}
\begin{align}\label{conclude3}
I'(\tilde{z}_{n}^{1})&=I'(u_{n})-I'(\bar{u})+o(1).
\end{align}
Then $\tilde{z}_{n}^{1}$ is a Palais-Smale sequence for $I$. Since $z_{n}^{1}\rightharpoonup 0$ in $D^{s,2}(\mathbb{R}^{3})$ and $z_{n}^{1}\nrightarrow 0$, we have that
$\tilde{z}_{n}^{1}\rightharpoonup 0$ in $D^{s,2}(\mathbb{R}^{3})$ and $\tilde{z}_{n}^{1}\nrightarrow 0$. Thus $\tilde{z}_{n}^{1}$ satisfies the assumptions of Lemma \ref{unps}, so there exist a sequence of points $\{x_{n}^{1}\}$ and a sequence of positive numbers $\eta_{n}^{1}$ such that
\begin{align*}
v_{n}^{1}(x)=(\eta_{n}^{1})^{\frac{3-2s}{2}}\tilde{z}_{n}^{1}(\eta_{n}^{1}x+x_{n}^{1})
\end{align*}
converges weakly in $D^{s,2}(\mathbb{R}^{3})$ to a nontrival solution $u^{1}$ of (\ref{linear}) and by (\ref{iun1}), (\ref{iun2}), (\ref{conclude1}) and (\ref{conclude2}), we have
\begin{align}\label{vnuz1}
I_{\infty}(v_{n}^{1}-u^{1})=I(\tilde{z}_{n}^{1})-I_{\infty}(u^{1})+o(1)=I(u_{n})-I(\bar{u})-I_{\infty}(u^{1})+o(1)
\end{align}
\begin{align}\label{vnuz2}
\|v_{m}^{1}-u^{1}\|_{D^{s,2}}^{2}=\|\tilde{z}_{n}^{1}\|_{D^{s,2}}^{2}-\|u^{1}\|_{D^{s,2}}^{2}+o(1)=\|u_{n}\|_{D^{s,2}}^{2}-\|\bar{u}\|_{D^{s,2}}^{2}-\|u^{1}\|_{D^{s,2}}^{2}+o(1).
\end{align}
Iterating this procedure, we get Palais-Smale sequences of functions
\begin{align*}
z_{n}^{j}=v_{n}^{j-1}-u^{j-1}, \ z_{n}^{j}\rightharpoonup 0 \ \mbox{in}  \ D^{s,2}(\mathbb{R}^{3}),
\end{align*}
and
\begin{align*}
\tilde{z}_{n}^{j}\in C_{0}^{\infty}(\mathbb{R}^{3}), \ \tilde{z}_{n}^{j}\rightharpoonup 0 \ \mbox{in}  \ D^{s,2}(\mathbb{R}^{3})
\end{align*}
such that
$z_{n}^{j}=\tilde{z}_{n}^{j}+(z_{n}^{j}-\tilde{z}_{n}^{j})$ and
\begin{align}\label{znj}
z_{n}^{j}-\tilde{z}_{n}^{j}\rightarrow 0 \ \ \mbox{in} \ \ D^{s,2}(\mathbb{R}^{3}).
\end{align}
We can obtain sequences of points $\{x_{n}\}^{j}\subset \mathbb{R}^{3}$ and sequences of numbers $\eta_{n}^{j}$ such that
\begin{align*}
v_{n}^{j}(x)=(\eta_{n}^{j})^{\frac{3-2s}{2}}\tilde{z}_{n}^{j}(\eta_{n}^{j}x+x_{n}^{j})
\end{align*}
converges weakly in $D^{s,2}(\mathbb{R}^{3})$ to a nontrivial solution $u^{j}$ of (\ref{linear}).
Furthermore, by (\ref{vnuz1}), (\ref{vnuz2}) and (\ref{znj}), we have
\begin{align*}
I_{\infty}(v_{n}^{j})&=I_{\infty}(\tilde{z}_{n}^{j})=I_{\infty}(v_{n}^{j-1}-u^{j-1})+o(1)\\
&=I(u_{n})-I(\bar{u})-\sum_{i=1}^{j-1}I_{\infty}(u^{i})+o(1)
\end{align*}
and
\begin{align}\label{vnjd}
\|v_{n}^{j}\|_{D^{s,2}}^{2}&=\|\tilde{z}_{n}^{j}\|_{D^{s,2}}^{2}=\|v_{n}^{j-1}-u^{j-1}\|_{D^{s,2}}^{2}+o(1)=\|v_{n}\|_{D^{s,2}}^{2}-\|u^{j-1}\|_{D^{s,2}}^{2}+o(1)\nonumber\\
&=\|u_{n}\|_{D^{s,2}}^{2}-\|\bar{u}\|_{D^{s,2}}^{2}-\sum_{i=1}^{j-1}\|u^{i}\|_{D^{s,2}}^{2}+o(1).
\end{align}
By the definition of $S_{s}$, we have $\|u^{j}\|_{D^{s,2}}^{2}\geq S_{s}\|u^{j}\|_{2_{s}^{\ast}}^{2}$. Since $u^{j}$ is a solution of (\ref{linear}), then we have
\begin{align}\label{ujd}
\|u^{j}\|_{D^{s,2}}^{2}\geq S_{s}^{\frac{3}{2s}}.
\end{align}
This together with the boundedness of $u_{n}$ and (\ref{vnjd}) yield that the iteration must terminate at some $k>0$. The proof is complete.
\end{proof}
\begin{corollary}\label{und}
Assume that $\{u_{n}\}\subset D^{s,2}(\mathbb{R}^{3})$ satisfies the assumption of the Theorem \ref{compactness} with $c\in (\frac{s}{3}S_{s}^{\frac{3}{2s}}, \frac{2s}{3}S_{s}^{\frac{3}{2s}})$, then
$\{u_{n}\}$ contains a subsequence strongly convergent in $D^{s,2}(\mathbb{R}^{3})$.
\end{corollary}
\begin{proof}
By Lemma \ref{minimization}, any nontrival solution $u$ have energy $I(u)>\frac{s}{3}S_{s}^{\frac{3}{2s}}$. From (3.15) in \cite{Cora}, any sign-changing solution $v$ of (\ref{linear}) satisfies
$I_{\infty}(v)\geq \frac{2s}{3}S_{s}^{\frac{3}{2s}}$. Since any positive solution of (\ref{linear}) has energy $\frac{s}{3}S_{s}^{\frac{3}{2s}}$, by Theorem \ref{compactness}, we get the conclusion.
\end{proof}
\begin{corollary}\label{unm}
If $\{u_{n}\}$ is a minimizing sequence for $I$ on $N$, then there exist a sequence of points $\{y_{n}\}\subset \mathbb{R}^{3}$, a sequence of positive numbers $\{\delta_{n}\}\subset \mathbb{R}^{+}$ and a sequence
$\{w_{n}\}\subset D^{s,2}(\mathbb{R}^{3})$ such that
\begin{align*}
u_{n}(x)=w_{n}(x)+\psi_{\delta_{n},y_{n}}(x),
\end{align*}
where $\psi_{\delta_{n},y_{n}}(x)$ are functions defined in (\ref{psi}) and $w_{n}\rightarrow 0$ strongly in $D^{s,2}(\mathbb{R}^{3})$.
\end{corollary}
\begin{proof}
The result follows from Theorem \ref{compactness} and Lemma \ref{minimization}.
\end{proof}

\section{\textbf{Proof of the main result}}
Set
\begin{align*}
\sigma(x)=\left\{\begin{aligned}
&0, && \mbox{if}\ |x|<1\\
&1, && \mbox{if}\ |x|\geq1
\end{aligned}\right.
\end{align*}
and define
\begin{align*}
\alpha: D^{s,2}(\mathbb{R}^{3})\rightarrow \mathbb{R}^{3}\times \mathbb{R^{+}}
\end{align*}
\begin{align*}
\alpha(u)=\frac{1}{S_{s}^{\frac{3}{2s}}}\int_{\mathbb{R}^{3}}\left(\frac{x}{|x|},\sigma(x)\right)|(-\Delta)^{\frac{s}{2}}u|^{2}dx=(\beta(u),\gamma(u)),
\end{align*}
where
\begin{align*}
\beta(u)=\frac{1}{S_{s}^{\frac{3}{2s}}}\int_{\mathbb{R}^{3}}\frac{x}{|x|}|(-\Delta)^{\frac{s}{2}}u|^{2}dx, \ \ \gamma(u)=\frac{1}{S_{s}^{\frac{3}{2s}}}\int_{\mathbb{R}^{3}}\sigma(x)|(-\Delta)^{\frac{s}{2}}u|^{2}dx.
\end{align*}
\begin{lemma}\label{half}
If $|y|\geq \frac{1}{2}$, then
\begin{align*}
\beta(\psi_{\delta,y})=\frac{y}{|y|}+o(1) \ \  \mbox{as} \ \delta\rightarrow 0.
\end{align*}
\end{lemma}
\begin{proof}
From (\ref{psi}), it follows that
\begin{align*}
|\nabla \psi_{\delta,y}(x)|=(3-2s)b_{s}\delta^{\frac{3-2s}{2}}\frac{|x-y|}{(\delta^{2}+|x-y|^{2})^{\frac{5-2s}{2}}}.
\end{align*}
By (\ref{equivalent}) and Proposition 2.2 in \cite{DPV}, we have
\begin{align}\label{deltay}
\int_{\mathbb{R}^{3}\backslash B_{\varepsilon}(y)}|(-\Delta)^{\frac{s}{2}}\psi_{\delta,y}|^{2}dx&\leq C\int_{\mathbb{R}^{3}\backslash B_{\varepsilon}(y)}|\nabla\psi_{\delta,y}|^{2}dx\nonumber\\
&=(3-2s)b_{s}C\delta^{\frac{3-2s}{2}}\int_{\mathbb{R}^{3}\backslash B_{\varepsilon}(y)}\frac{|x-y|^{2}}{(\delta^{2}+|x-y|^{2})^{5-2s}}dx\nonumber\\
&=C_{1}\delta^{\frac{3-2s}{2}}\int_{\varepsilon}^{+\infty}\frac{\rho^{4}}{(\delta^{2}+\rho^{2})^{5-2s}}d\rho\nonumber\\
&\leq C_{1}\delta^{\frac{3-2s}{2}}\int_{\varepsilon}^{+\infty}\rho^{4s-6}d\rho.
\end{align}
By (\ref{deltay}),  for every $\varepsilon>0$, there exists a $\hat{\delta}$ such that $\forall \delta\in (0,\hat{\delta}]$,
\begin{align*}
\frac{1}{S_{s}^{\frac{3}{2s}}}\int_{\mathbb{R}^{3}\backslash B_{\varepsilon}(y)}|(-\Delta)^{\frac{s}{2}}\psi_{\delta,y}|^{2}dx<\varepsilon.
\end{align*}
Then
\begin{align}\label{minus1}
\left|\beta(\psi_{\delta,y})-\frac{1}{S_{s}^{\frac{3}{2s}}}\int_{B_{\varepsilon}(y)}\frac{x}{|x|}|(-\Delta)^{\frac{s}{2}}\psi_{\delta,y}|^{2}dx\right|<\varepsilon.
\end{align}
If $\varepsilon$ is small enough, for $|y|\geq \frac{1}{2}$ and $x\in B_{\varepsilon}(y)$,
\begin{align*}
\left|\frac{x}{|x|}-\frac{y}{|y|}\right|<2\varepsilon.
\end{align*}
Then by (\ref{energy}),
\begin{align}\label{minus2}
&\left|\frac{y}{|y|}-\frac{1}{S_{s}^{\frac{3}{2s}}}\int_{B_{\varepsilon}(y)}\frac{x}{|x|}|(-\Delta)^{\frac{s}{2}}\psi_{\delta,y}(x)|^{2}dx\right|\nonumber\\
&=\left|\frac{1}{S_{s}^{\frac{3}{2s}}}\left(\frac{y}{|y|}-\frac{x}{|x|}\right)|(-\Delta)^{\frac{s}{2}}\psi_{\delta,y}(x)|^{2}dx+\frac{1}{S_{s}^{\frac{3}{2s}}}\int_{\mathbb{R}^{3}\backslash B_{\varepsilon}(y)}\frac{y}{|y|}|(-\Delta)^{\frac{s}{2}}\psi_{\delta,y}(x)|^{2}dx\right|\nonumber\\
&\leq \frac{2\varepsilon}{S_{s}^{\frac{3}{2s}}}|(-\Delta)^{\frac{s}{2}}\psi_{\delta,y}(x)|^{2}dx+\varepsilon=3\varepsilon.
\end{align}
By (\ref{minus1}) and (\ref{minus2}), we have
\begin{align*}
\left|\beta(\psi_{\delta,y})-\frac{y}{|y|}\right|<4\varepsilon.
\end{align*}
\end{proof}
Define
\begin{align*}
\mathcal{M}=\left\{u\in \mathcal{N}:\alpha(u)=(\beta(u),\gamma(u))=(0,\frac{1}{2})\right\},
\end{align*}
and define
\begin{align*}
c_{\mathcal{M}}=\inf_{u\in \mathcal{M}}I(u).
\end{align*}
\begin{lemma}\label{vkc}
Suppose that $V\geq 0$, $V\in L^{\frac{3}{2s}}(\mathbb{R}^{3})$, $K\geq 0$, $K\in L^{\frac{6}{6s-3}}(\mathbb{R}^{3})$ and $\|V\|_{\frac{3}{2s}}+\|K\|_{\frac{6}{6s-3}}>0$. Then
\begin{align*}
c_{\mathcal{M}}>\frac{s}{3}S_{s}^{\frac{3}{2s}}.
\end{align*}
\end{lemma}
\begin{proof}
It is obvious that $c_{0}\geq \frac{s}{3}S_{s}^{\frac{3}{2s}}$. To prove $c_{0}>\frac{s}{3}S_{s}^{\frac{3}{2s}}$, we argue by contradiction. Suppose that there exists a sequence
$\{u_{n}\}\subset \mathcal{N}$ such that
\begin{align}\label{betagamma}
\beta(u_{n})=0, \ \gamma(u_{n})=\frac{1}{2}
\end{align}
\begin{align}\label{iunm}
\lim_{n\rightarrow\infty}I(u_{n})=\frac{s}{3}S_{s}^{\frac{3}{2s}}.
\end{align}
By Corollary \ref{unm}, there exist a sequence of points $\{y_{n}\}\subset \mathbb{R}^{3}$, a sequence of positive numbers $\{\delta_{n}\}\subset \mathbb{R}^{+}$ and a sequence
of functions $\{w_{n}\}\subset D^{s,2}(\mathbb{R}^{3})$ converging to 0 in $D^{s,2}(\mathbb{R}^{3})$ such that
\begin{align*}
u_{n}(x)=w_{n}(x)+\psi_{\delta_{n},y_{n}}(x).
\end{align*}
Since $w_{n}\rightarrow 0$ in $D^{s,2}(\mathbb{R}^{3})$, then for $n$ big enough, we have
\begin{align}\label{wnpsi}
\alpha(w_{n}+\psi_{\delta_{n},y_{n}})&=\frac{1}{S_{s}^{\frac{3}{2}}}\int_{\mathbb{R}^{3}}\left(\frac{x}{|x|},\sigma(x)\right)|(-\Delta)^{\frac{s}{2}}(w_{n}+\psi_{\delta_{n},y_{n}})|^{2}dx\nonumber\\
&=\frac{1}{S_{s}^{\frac{3}{2}}}\int_{\mathbb{R}^{3}}\left(\frac{x}{|x|},\sigma(x)\right)|(-\Delta)^{\frac{s}{2}}w_{n}|^{2}dx\nonumber\\
&\quad+\frac{2}{S_{s}^{\frac{3}{2}}}\int_{\mathbb{R}^{3}}\left(\frac{x}{|x|},\sigma(x)\right)((-\Delta)^{\frac{s}{2}}w_{n},(-\Delta)^{\frac{s}{2}}\psi_{\delta_{n},y_{n}})dx\nonumber\\
&\quad+\frac{1}{S_{s}^{\frac{3}{2}}}\int_{\mathbb{R}^{3}}\left(\frac{x}{|x|},\sigma(x)\right)|(-\Delta)^{\frac{s}{2}}\psi_{\delta_{n},y_{n}}|^{2}dx\nonumber\\
&=\frac{1}{S_{s}^{\frac{3}{2}}}\int_{\mathbb{R}^{3}}\left(\frac{x}{|x|},\sigma(x)\right)|(-\Delta)^{\frac{s}{2}}\psi_{\delta_{n},y_{n}}|^{2}dx+o(1)\nonumber\\
&=\alpha(\psi_{\delta_{n},y_{n}})+o(1).
\end{align}
From (\ref{betagamma}) and (\ref{wnpsi}), it follows that
\begin{align*}
\mbox{(i)} \ \ \beta(\psi_{\delta_{n},y_{n}})\rightarrow 0 \ \ \mbox{as}  \ \ n\rightarrow \infty.
\end{align*}
\begin{align*}
\mbox{(ii)} \ \ \gamma(\psi_{\delta_{n},y_{n}})\rightarrow 0 \ \ \mbox{as} \ n\rightarrow \infty.
\end{align*}
For $\delta_{n}$, up to a sunsequence, one of these cases occurs\\
(a) $\delta_{n}\rightarrow \infty$ as $n\rightarrow \infty$\\
(b) $\delta_{n}\rightarrow \tilde{\delta}\neq 0$ as $n\rightarrow \infty$\\
(c) $\delta_{n}\rightarrow 0$ and $y_{n}\rightarrow \bar{y}$, $\bar{y}<\frac{1}{2}$ as $n\rightarrow \infty$\\
(d) $\delta_{n}\rightarrow 0$ as $n\rightarrow \infty$ and $|y_{n}|\geq \frac{1}{2}$ for $n$ large. \\
We will prove that none of the possibilities (a)-(d) can be true. If (a) holds,
\begin{align*}
\gamma(\psi_{\delta_{n},y_{n}})&=\frac{1}{S_{s}^{\frac{3}{2s}}}\int_{\mathbb{R}^{3}}\sigma(x)|(-\triangle)^{\frac{s}{2}}\psi_{\delta_{n},y_{n}}|^{2}dx=\frac{1}{S_{s}^{\frac{3}{2s}}}\int_{\mathbb{R}^{3}\backslash B_{1}(0) }|(-\triangle)^{\frac{s}{2}}\psi_{\delta_{n},y_{n}}|^{2}dx\\
&=1-\frac{1}{S_{s}^{\frac{3}{2s}}}\int_{B_{1}(0)}|(-\triangle)^{\frac{s}{2}}\psi_{\delta_{n},y_{n}}|^{2}dx=1-o(1) \ \  \ \ \mbox{as} \ n\rightarrow \infty.
\end{align*}
This contradicts  (ii). If (b) holds, then $|y_{n}|\rightarrow +\infty$, otherwise $\psi_{\delta_{n},y_{n}}$ would converge strongly in $D^{s,2}(\mathbb{R}^{3})$, so $u_{n}$ would converge strongly in $D^{s,2}(\mathbb{R}^{3})$ against Lemma \ref{minimization}. Then we have
\begin{align*}
\gamma(\psi_{\delta_{n},y_{n}})&=\gamma(\psi_{\bar{\delta},y_{n}})+o(1)=\frac{1}{S_{s}^{\frac{3}{2s}}}\int_{\mathbb{R}^{3}}\sigma(x)|(-\Delta)^{\frac{s}{2}}\psi_{\bar{\delta},y_{n}}|^{2}dx+o(1)\\
&=\frac{1}{S_{s}^{\frac{3}{2s}}}\int_{\mathbb{R}^{3}}\sigma(x-y_{n})|(-\Delta)^{\frac{s}{2}}\psi_{\bar{\delta},0}|^{2}dx+o(1)\\
&=1-\frac{1}{S_{s}^{\frac{3}{2s}}}\int_{B_{1}(y_{n})}|(-\Delta)^{\frac{s}{2}}\psi_{\bar{\delta},0}|^{2}dx+o(1)\\
&=1+o(1) \ \ \ \  \mbox{as} \ n\rightarrow \infty.
\end{align*}
This contradicts (ii). If (c) holds, then
\begin{align*}
\gamma(\psi_{\delta_{n},y_{n}})&=\frac{1}{S_{s}^{\frac{3}{2s}}}\int_{\mathbb{R}^{3}\backslash B_{1}(0)}|(-\Delta)^{\frac{s}{2}}\psi_{\bar{\delta},y_{n}}|^{2}dx\\
&=\frac{1}{S_{s}^{\frac{3}{2s}}}\int_{\mathbb{R}^{3}\backslash B_{1}(y_{n})}|(-\Delta)^{\frac{s}{2}}\psi_{\bar{\delta},0}|^{2}dx\\
&=o(1).
\end{align*}
This contradicts (ii). If (d) holds, by Lemma \ref{half}, we have
\begin{align*}
\beta(\psi_{\delta_{n},y_{n}})=\frac{y_{n}}{|y_{n}|}+o(1) \  \ \ \mbox{as} \ n\rightarrow \infty.
\end{align*}
This contradicts (i).
\end{proof}

Define $\theta:D^{s,2}(\mathbb{R}^{3})\backslash \{0\}\rightarrow \mathcal{N}$:
\begin{align*}
\theta(u)=t(u)u,
\end{align*}
where $t(u)$ is the unique positive number such that $t(u)u\in \mathcal{N}$. Let $T$ be the operator
\begin{align*}
T:\mathbb{R}^{3}\times (0,+\infty)\rightarrow D^{s,2}(\mathbb{R}^{3})
\end{align*}
defined by
\begin{align*}
T(y,\delta)=\psi_{\delta,y}(x).
\end{align*}
\begin{lemma}\label{vkdelta}
Assume that $V\geq 0$, $V\in L^{\frac{3}{2s}}(\mathbb{R}^{3})$, $K\geq 0$, $K\in L^{\frac{6}{6s-3}}(\mathbb{R}^{3})$. Then for any $\varepsilon>0$, there exist $\delta_{1}=\delta_{1}(\varepsilon)$,  $\delta_{2}=\delta_{2}(\varepsilon)$ such that
\begin{align*}
I(\theta\circ T(y,\delta))<\frac{s}{3}S_{s}^{\frac{3}{2s}}+\varepsilon
\end{align*}
for any $y\in \mathbb{R}^{3}$ and $\delta\in (0,\delta_{1}]\cup [\delta_{2},+\infty)$.
\end{lemma}
\begin{proof}
Since $V\in L^{\frac{3}{2s}}(\mathbb{R}^{3})$, then for any $\varepsilon>0$, there exists $R>0$ such that
\begin{align}\label{rbv}
\left(\int_{\mathbb{R}^{3}\backslash B_{R}(0)}|V(x)|^{\frac{3}{2s}}dx\right)^{\frac{2s}{3}}<\frac{\varepsilon}{2S_{s}^{\frac{3-2s}{2s}}}.
\end{align}
From $\lim_{\delta\rightarrow +\infty}\sup_{y\in \mathbb{R}^{3}}|\psi_{\delta,y}|=0$, it follows that $\lim_{\delta\rightarrow +\infty}\int_{B_{R}(0)}|\psi_{\delta,y}|^{2_{s}^{\ast}}dx=0$.
Therefore there exists $\delta_{2}=\delta_{2}(\varepsilon)$ such that
\begin{align}\label{supdelta}
\sup_{y\in \mathbb{R}^{3}}\left(\int_{B_{R}(0)}|\psi_{\delta,y}|^{2_{s}^{\ast}}dx\right)^{\frac{3-2s}{3}}<\frac{\varepsilon}{2\|V\|_{\frac{3}{2s}}}, \ \ \ \delta\geq \delta_{2}.
\end{align}
By (\ref{rbv}) and (\ref{supdelta}), we have
\begin{align}\label{vpsi}
\int_{\mathbb{R}^{3}}V(x)\psi_{\delta,y}^{2}dx&=\int_{B_{R}(0)}V(x)\psi_{\delta,y}^{2}dx+\int_{\mathbb{R}^{3}\backslash B_{R}(0)}V(x)\psi_{\delta,y}^{2}dx\nonumber\\
&\leq \left(\int_{B_{R}(0)}|V(x)|^{\frac{3}{2s}}dx\right)^{\frac{2s}{3}}\left(\int_{B_{R}(0)}|\psi_{\delta,y}|^{2_{s}^{\ast}}dx\right)^{\frac{3-2s}{3}}\nonumber\\
&\quad+\left(\int_{\mathbb{R}^{3}\backslash B_{R}(0)}|V(x)|^{\frac{3}{2s}}dx\right)^{\frac{2s}{3}}\left(\int_{\mathbb{R}^{3}\backslash B_{R}(0)}|\psi_{\delta,y}|^{2_{s}^{\ast}}dx\right)^{\frac{3-2s}{3}}\nonumber\\
&\leq \|V\|_{\frac{3}{2s}}\left(\int_{B_{R}(0)}|\psi_{\delta,y}|^{2_{s}^{\ast}}dx\right)^{\frac{3-2s}{3}}+\left(\int_{\mathbb{R}^{3}\backslash B_{R}(0)}|V(x)|^{\frac{3}{2s}}dx\right)^{\frac{2s}{3}}\|\psi_{\delta,y}\|_{2_{s}^{\ast}}^{2}\nonumber\\
&<\frac{\varepsilon}{2}+\frac{\varepsilon}{2}=\varepsilon.
\end{align}
Since $V\in L^{\frac{3}{2s}}(\mathbb{R}^{3})$, for any $\varepsilon>0$, there exists $r>0$ small enough such that
\begin{align}\label{supy}
\sup_{y\in \mathbb{R}^{3}}\left(\int_{B_{r}(y)}|V(x)|^{\frac{3}{2s}}dx\right)^{\frac{2s}{3}}<\frac{\varepsilon}{2S_{s}^{\frac{3-2s}{2s}}}.
\end{align}
By (\ref{psi}),
\begin{align*}
\int_{\mathbb{R}^{3}\backslash B_{r}(y)}|\psi_{\delta,y}|^{2_{s}^{\ast}}dx&=C\delta^{3}\int_{\mathbb{R}^{3}\backslash B_{r}(y)}\frac{1}{(\delta^{2}+|x-y|^{2})^{3}}dx\\
&=\widetilde{C}\delta^{3}\int_{r}^{+\infty}\frac{\rho^{2}}{(\delta^{2}+\rho^{2})^{3}}d\rho,
\end{align*}
so there is a $\delta_{1}=\delta_{1}(\varepsilon)$ such that for any $\delta\in (0,\delta_{1}]$, there holds
\begin{align}\label{rbry}
\left(\int_{\mathbb{R}^{3}\backslash B_{r}(y)}|\psi_{\delta,y}|^{2_{s}^{\ast}}dx\right)^{\frac{3-2s}{3}}<\frac{\varepsilon}{2\|V\|_{\frac{3}{2s}}}, \ \ \ 0<\delta\leq \delta_{1}.
\end{align}
By (\ref{supy}) and (\ref{rbry}), we get
\begin{align}\label{vpsidelta}
\int_{\mathbb{R}^{3}}V(x)\psi_{\delta,y}^{2}dx&=\int_{B_{r}(y)}V(x)\psi_{\delta,y}^{2}dx+\int_{\mathbb{R}^{3}\backslash B_{r}(y)}V(x)\psi_{\delta,y}^{2}dx\nonumber\\
&\leq \left(\int_{B_{r}(y)}|V(x)|^{\frac{3}{2s}}dx\right)^{\frac{2s}{3}}\left(\int_{B_{r}(y)}|\psi_{\delta,y}|^{2_{s}^{\ast}}dx\right)^{\frac{3-2s}{3}}\nonumber\\
&\quad+\left(\int_{\mathbb{R}^{3}\backslash B_{r}(y)}|V(x)|^{\frac{3}{2s}}dx\right)^{\frac{2s}{3}}\left(\int_{\mathbb{R}^{3}\backslash B_{r}(y)}|\psi_{\delta,y}|^{2_{s}^{\ast}}dx\right)^{\frac{3-2s}{3}}\nonumber\\
&\leq \left(\int_{B_{r}(y)}|V(x)|^{\frac{3}{2s}}dx\right)^{\frac{2s}{3}}\left(\int_{B_{r}(y)}|\psi_{\delta,y}|^{2_{s}^{\ast}}dx\right)^{\frac{3-2s}{3}}\nonumber\\
&\quad+\|V\|_{\frac{3}{2s}}\left(\int_{\mathbb{R}^{3}\backslash B_{r}(y)}|V(x)|^{\frac{3}{2s}}dx\right)^{\frac{2s}{3}}\nonumber\\
&<\frac{\varepsilon}{2}+\frac{\varepsilon}{2}=\varepsilon.
\end{align}
From (\ref{vpsi}) and (\ref{vpsidelta}), it follows that
\begin{align}\label{rvdy}
\int_{\mathbb{R}^{3}}V(x)\psi_{\delta,y}^{2}dx<\varepsilon, \ \mbox{for any} \ y\in \mathbb{R}^{3} \ \mbox{and} \ \ \delta\in (0,\delta_{1}]\cup [\delta_{2},+\infty).
\end{align}
Similarly, we can get that
\begin{align}\label{npsidelta}
N(\psi_{\delta,y})<\varepsilon, \ \mbox{for any} \ y\in \mathbb{R}^{3} \ \mbox{and} \ \ \delta\in (0,\delta_{1}]\cup [\delta_{2},+\infty).
\end{align}
Since $\psi_{\delta,y}\in \mathcal{N}_{\infty}$, it is easy to show that there exists $t_{\delta,y}=t(\psi_{\delta,y})$ such that $t_{\delta,y}\psi_{\delta,y}\in \mathcal{N}$. By an argument analogous to (\ref{nw})-(\ref{imply}), we verify that for any $y\in \mathbb{R}^{3}$,
\begin{align}\label{tdelta}
t_{\delta,y}\rightarrow 1 \ \ \ \mbox{as} \ \ \delta\rightarrow 0 \ \ \mbox{or} \ \ \delta\rightarrow+\infty.
\end{align}
By (\ref{rvdy}), (\ref{npsidelta}) and (\ref{tdelta}), we obtain
\begin{align*}
I(\theta\circ T(y,\delta))&=\frac{1}{2}t_{\delta,y}^{2}\|\psi_{\delta,y}\|_{D^{s,2}}^{2}+\frac{1}{2}t_{\delta,y}^{2}\int_{\mathbb{R}^{3}}V(x)\psi_{\delta,y}^{2}dx+\frac{1}{4}t_{\delta,y}^{4}N(\psi_{\delta,y})\\
&\quad-\frac{1}{2_{s}^{\ast}}t_{\delta,y}^{2_{s}^{\ast}}\|\psi_{\delta,y}\|_{2_{s}^{\ast}}^{2_{s}^{\ast}}\\
&=I_{\infty}(t_{\delta,y}\psi_{\delta,y})+\frac{1}{2}t_{\delta,y}^{2}\int_{\mathbb{R}^{3}}V(x)\psi_{\delta,y}^{2}dx+\frac{1}{4}t_{\delta,y}^{4}N(\psi_{\delta,y})\\
&<I_{\infty}(t_{\delta,y}\psi_{\delta,y})+\varepsilon\\
&=\frac{s}{3}S_{s}^{\frac{3}{2s}}+\varepsilon,
\end{align*}
for any $y\in \mathbb{R}^{3}$ and $\delta\in (0,\delta_{1}]\cup [\delta_{2},+\infty)$.
\end{proof}
\begin{lemma}\label{vky}
Assume that $V\geq 0$, $V\in L^{\frac{3}{2s}}(\mathbb{R}^{3})$, $K\geq 0$, $K\in L^{\frac{6}{6s-3}}(\mathbb{R}^{3})$. Then for any fixed $\delta>0$,
\begin{align*}
\lim_{|y|\rightarrow +\infty}I(\theta\circ T(y,\delta))=\frac{s}{3}S_{s}^{\frac{3}{2s}}.
\end{align*}
\end{lemma}
\begin{proof}
It is obvious that $\lim_{|y|\rightarrow +\infty}\int_{\mathbb{R}^{3}}V(x)\psi_{\delta,y}^{2}dx=0$, $\lim_{|y|\rightarrow +\infty}N(\psi_{\delta,y})=0$.
By a similar argument to (\ref{iun})-(\ref{imply}), we have
\begin{align*}
t_{\delta,y}=t(\psi_{\delta,y})\rightarrow 1 \ \ \mbox{as} \ \ |y|\rightarrow +\infty,
\end{align*}
where $t_{\delta,y}$ is the unique positive number such that $t_{\delta,y}\psi_{\delta,y}\in \mathcal{N}$. Then
\begin{align*}
\frac{s}{3}S_{s}^{\frac{3}{2s}}&\leq I(\theta\circ T(y,\delta))\\
&=\frac{1}{2}t_{\delta,y}^{2}\|\psi_{\delta,y}\|_{D^{s,2}}^{2}+\frac{1}{2}t_{\delta,y}^{2}\int_{\mathbb{R}^{3}}V(x)\psi_{\delta,y}^{2}dx+\frac{1}{4}t_{\delta,y}^{4}N(\psi_{\delta,y})\\
&\quad-\frac{1}{2_{s}^{\ast}}t_{\delta,y}^{2_{s}^{\ast}}\|\psi_{\delta,y}\|_{2_{s}^{\ast}}^{2_{s}^{\ast}}\\
&<I_{\infty}(\psi_{\delta,y})+o(1)\\
&=\frac{s}{3}S_{s}^{\frac{3}{2s}}+o(1) \ \ \ \mbox{as} \ \ \ |y|\rightarrow +\infty.
\end{align*}
\end{proof}

By Lemma \ref{vkc}, there exist a positive number $\mu$ such that $\frac{s}{3}S_{s}^{\frac{3}{2s}}+\mu<c_{0}$.
\begin{lemma}\label{deltatheta}
There is a $\delta_{1}:0<\delta<\frac{1}{2}$ such that \\
$(a)$ $I(\theta\circ T(y,\delta))<\frac{s}{3}S_{s}^{\frac{3}{2s}}+\mu$, \ \ $\forall y\in \mathbb{R}^{3}$\\
$(b)$ $\gamma(\theta\circ T(y,\delta))<\frac{1}{2}$, \ \ $\forall y: |y|<\frac{1}{2}$\\
$(c)$ $\big|\beta(\theta\circ T(y,\delta))-\frac{y}{|y|}\big|<\frac{1}{4}$, \ \ $\forall y: |y|\geq \frac{1}{2}$.
\end{lemma}
\begin{proof}
By Lemma \ref{vkdelta}, $(a)$ holds. By (\ref{equivalent}) and Proposition 2.2 in \cite{DPV},
\begin{align*}
\gamma(\theta\circ T(y,\delta))&=\frac{t_{\delta,y}^{2}}{S_{s}^{\frac{3}{2s}}}\int_{\mathbb{R}^{3}}\sigma(x)|(-\Delta)^{\frac{s}{2}}\psi_{\delta,y}|^{2}dx\\
&=\frac{t_{\delta,y}^{2}}{S_{s}^{\frac{3}{2s}}}\int_{\mathbb{R}^{3}\backslash B_{1}(0)}|(-\Delta)^{\frac{s}{2}}\psi_{\delta,y}|^{2}dx\\
&=\frac{t_{\delta,y}^{2}}{S_{s}^{\frac{3}{2s}}}\int_{\mathbb{R}^{3}\backslash B_{1}(y)}|(-\Delta)^{\frac{s}{2}}\psi_{\delta,0}|^{2}dx\\
&\leq Ct_{\delta,y}^{2}\int_{\mathbb{R}^{3}\backslash B_{1}(y)}|\nabla \psi_{\delta,0}|^{2}dx\\
&=\tilde{C}t_{\delta,y}^{2}\delta^{\frac{3-2s}{2}}\int_{\mathbb{R}^{3}\backslash B_{1}(y)}\frac{|x|^{2}}{(\delta^{2}+|x|^{2})^{5-2s}}dx,
\end{align*}
this together with (\ref{tdelta}) yield that $(b)$ holds. By Lemma \ref{half} and (\ref{tdelta}), we get $(c)$.
\end{proof}
\begin{lemma}\label{deltagamma}
There exists a $\delta_{2}>\frac{1}{2}$ such that \\
$(a)$ $I(\theta\circ T(y,\delta))<\frac{s}{3}S_{s}^{\frac{3}{2s}}+\mu$, \ \ $\forall y\in \mathbb{R}^{3}$\\
$(b)$ $\gamma(\theta\circ T(y,\delta))>\frac{1}{2}$, \ \ $\forall y: |y|>\frac{1}{2}$.
\end{lemma}
\begin{proof}
By Lemma \ref{vkdelta}, $(a)$ holds. Since
\begin{align*}
\lim_{\delta\rightarrow +\infty}\int_{B_{1}(0)}|(-\Delta)^{\frac{s}{2}}\psi_{\delta,y}|^{2}dx=0,
\end{align*}
and
\begin{align*}
\lim_{\delta\rightarrow +\infty}t_{\delta,y}=1,
\end{align*}
we have
\begin{align*}
\gamma(\theta\circ T(y,\delta))=t_{\delta,y}\left(1-\frac{1}{S_{s}^{\frac{3}{2s}}}\int_{B_{1}(0)}|(-\Delta)^{\frac{s}{2}}\psi_{\delta,y}|^{2}dx\right)\rightarrow 1 \ \ \mbox{as} \ \ \delta\rightarrow +\infty,
\end{align*}
then $(b)$ holds.
\end{proof}
\begin{lemma}\label{rtheta}
There exists $R>0$ such that \\
$(a)$ $I(\theta\circ T(y,\delta))<\frac{s}{3}S_{s}^{\frac{3}{2s}}+\mu$, \ \ $\forall y:|y|\geq R$ and $\delta\in [\delta_{1},\delta_{2}]$\\
$(b)$ $(\beta(\theta\circ T(y,\delta))\cdot y)_{\mathbb{R}^{3}}>0$, \ \ $\forall |y|\geq R$ and $\delta\in [\delta_{1},\delta_{2}]$.
\end{lemma}
\begin{proof}
Since
\begin{align}\label{ty}
t_{\delta,y}\rightarrow 1 \ \mbox{as} \ |y|\rightarrow +\infty,
\end{align}
by Lemma \ref{vky} and the compactness of $[\delta_{1},\delta_{2}]$, we can get $R_{1}$ big enough such that
\begin{align*}
I(\theta\circ T(y,\delta))<\frac{s}{3}S_{s}^{\frac{3}{2s}}+\mu, \ \ \forall y:|y|\geq R \  \mbox{and} \ \delta\in [\delta_{1},\delta_{2}].
\end{align*}
For $|y|$ large enough, by (\ref{ty}) and a similar argument as that in the proof of Lemma \ref{half},  the result follows.
\end{proof}

Let $\delta_{1}$, $\delta_{2}$ and $R$ be the constant in Lemma \ref{deltatheta}, Lemma \ref{betagamma} and Lemma \ref{rtheta}, respectively. Define a bounded domain $D\subset \mathbb{R}^{3}\times \mathbb{R}$ by
\begin{align*}
D=\left\{(y,\delta)\in \mathbb{R}^{3}\times \mathbb{R}: |y|\leq R, \ \delta_{1}\leq \delta\leq \delta_{2}\right\},
\end{align*}
and define the map $\vartheta:D\rightarrow \mathbb{R}^{3}\times \mathbb{R}^{+}$ by
\begin{align*}
\vartheta(y,\delta)=\left(\beta\circ\theta\circ T(y,\delta),\gamma\circ\theta\circ T(y,\delta)\right).
\end{align*}
\begin{lemma}\label{topologicaldegree}
Assume that $V\geq 0$, $V\in L^{\frac{3}{2s}}(\mathbb{R}^{3})$, $K\geq 0$, $K\in L^{\frac{6}{6s-3}}(\mathbb{R}^{3})$ and $\|V\|_{\frac{3}{2s}}+\|K\|_{\frac{6}{6s-3}}>0$. Then
\begin{align*}
\emph{deg}\left(\vartheta,D,(0,\frac{1}{2})\right)=1.
\end{align*}
\end{lemma}
\begin{proof}
We consider the homotopy
\begin{align*}
\zeta(y,\delta,s)=(1-s)(y,\delta)+s\vartheta(y,\delta).
\end{align*}
By the homotopy invariance of the topological degree, and by the fact that
\begin{align*}
\emph{deg}\left(id,D,(0,\frac{1}{2})\right)=1,
\end{align*}
we need to prove that
\begin{align*}
\zeta(y,\delta,s)\neq (0,\frac{1}{2}) \ \mbox{for any} \ \ (y,\delta)\in \partial D \ \mbox{and} \ s\in [0,1].
\end{align*}
We have
\begin{align*}
\partial D=\Gamma_{1}\cup\Gamma_{2}\cup \Gamma_{3}\cup \Gamma_{4},
\end{align*}
where
\begin{align*}
\Gamma_{1}&=\{(y,\delta_{1}):|y|<\frac{1}{2}\}\\
\Gamma_{2}&=\{(y,\delta_{1}):\frac{1}{2}\leq|y|\leq R\}\\
\Gamma_{3}&=\{(y,\delta_{2}):|y|\leq R\}\\
\Gamma_{4}&=\{(y,\delta):|y|=R,\delta\in[\delta_{1},\delta_{2}]\}.
\end{align*}
If $(y,\delta)\in \Gamma_{1}$, by Lemma \ref{deltatheta} $(b)$,
\begin{align*}
(1-s)\delta_{1}+s\gamma\circ\theta\circ T(y,\delta_{1})<\frac{1}{2}.
\end{align*}
If $(y,\delta)\in \Gamma_{2}$, by Lemma \ref{deltatheta} $(c)$,
\begin{align*}
\big|\beta(\theta\circ T(y,\delta))-\frac{y}{|y|}\big|<\frac{1}{4},
\end{align*}
then $\forall s\in [0,1]$,
\begin{align*}
|(1-s)y+s\beta(\theta\circ T(y,\delta_{1}))|&\geq \left|(1-s)y+s\frac{y}{|y|}\right|-\left|s\beta(\theta\circ T(y,\delta_{1}))-s\frac{y}{|y|}\right|\\
&\geq (1-s)|y|+s-\frac{s}{4}\\
&\geq \frac{1}{2}.
\end{align*}
If $(y,\delta)\in \Gamma_{3}$,
\begin{align*}
(1-s)\delta_{2}+s\gamma\circ \theta \circ T(y,\delta_{2})>\frac{1}{2} \ \mbox{for any} \ s\in [0,1].
\end{align*}
If $(y,\delta)\in \Gamma_{4}$, by Lemma \ref{rtheta} $(b)$,
\begin{align*}
([(1-s)y+s\beta\circ \theta \circ T(y,\delta)]\cdot y)>0 \ \mbox{for any} \ s\in [0,1].
\end{align*}
Therefore,
\begin{align*}
deg(\vartheta,D,(0,\frac{1}{2}))=deg(id,D,(0,\frac{1}{2}))=1.
\end{align*}
\end{proof}

\textbf{Proof of Theorem \ref{boundstate}.} In order to apply the linking theorem in \cite{Struwe}, we define
\begin{align*}
Q=\theta\circ T(D), \ \mathcal{M}=\left\{u\in \mathcal{N}: \alpha(u)=(\beta(u),\gamma(u))=(0,\frac{1}{2})\right\}.
\end{align*}
We claim that $\mathcal{M}$ links $\partial Q$, that is, \\
$(a)$ $\partial Q\cap \mathcal{M}=\varnothing$;\\
$(b)$ $h(Q)\cap \mathcal{M}\neq \varnothing$ for any $h\in \Gamma=\{h\in C(Q,\mathcal{N}):h(\partial Q)=id\}$. \\
In fact, if $u\in \theta\circ T(\partial D)$, by Lemma \ref{deltatheta} $(a)$,  Lemma \ref{deltagamma} $(a)$,  Lemma \ref{rtheta} $(a)$,
\begin{align*}
I(u)<m+\mu<c_{\mathcal{M}},
\end{align*}
then $u\notin \mathcal{M}$.

To prove $(b)$, for any $h \in\Gamma$, we define $\eta: D\rightarrow \mathbb{R}^{3}\times \mathbb{R}+$ by
\begin{align*}
\eta(y,\delta)=\left(\beta\circ h\circ \theta\circ T(y,\delta),\gamma\circ h\circ \theta\circ T(y,\delta)\right).
\end{align*}
Since $h(\partial D)=id$, then
\begin{align*}
\eta(y,\delta)=\left(\beta\circ \theta\circ T(y,\delta),\gamma\circ \theta\circ T(y,\delta)\right)=\vartheta(y,\delta) \ \mbox{for any} \ (y,\delta)\in \partial D.
\end{align*}
This together with Lemma \ref{topologicaldegree} yield that
\begin{align*}
deg(\eta,D,(0,\frac{1}{2}))=deg(\vartheta,D,(0,\frac{1}{2}))=1.
\end{align*}
Then there exists $(y',\delta')\in D$ such that $h\circ \theta \circ T(y',\delta')\in \mathcal{M}$.

Define
\begin{align}\label{level}
d=\inf_{h\in \Gamma}\max_{u\in Q}I(h(u)).
\end{align}
By linking theorem, $d\geq c_{\mathcal{M}}>\frac{s}{3}S_{s}^{\frac{3}{2s}}$. From $Q=\theta\circ T(D)$ and (\ref{level}), it follows that
\begin{align*}
d\leq \max_{u\in Q}I(u)\leq \sup_{(\delta,y)\in D}I(t_{\delta,y}\psi_{\delta,y}).
\end{align*}
By $t_{\delta,y}\psi_{\delta,y}\in \mathcal{N}$, we have
\begin{align}\label{tdeltaypsi}
t_{\delta,y}^{2}\|\psi_{\delta,y}\|_{D^{s,2}}^{2}+t_{\delta,y}^{2}\int_{\mathbb{R}^{3}}V(x)\psi_{\delta,y}^{2}dx+t_{\delta,y}^{4}N(\psi_{\delta,y})
=t_{\delta,y}^{2_{s}^{\ast}}\|\psi_{\delta,y}\|_{2_{s}^{\ast}}^{2_{s}^{\ast}},
\end{align}
this together with $\|\psi_{\delta,y}\|_{2_{s}^{\ast}}^{2_{s}^{\ast}}=\|\psi_{\delta,y}\|_{D^{s,2}}^{2}$ yield that
\begin{align}\label{deduce}
(1-t_{\delta,y}^{2_{s}^{\ast}-2})\|\psi_{\delta,y}\|_{D^{s,2}}^{2}+\int_{\mathbb{R}^{3}}V(x)\psi_{\delta,y}^{2}dx+t_{\delta,y}^{2}N(\psi_{\delta,y})=0.
\end{align}
Since $\psi_{\delta,y}>0$, $V(x)\geq 0$, $K(x)\geq 0$ and $\|V\|_{\frac{3}{2s}}+\|K\|_{\frac{6}{6s-3}}>0$, by (\ref{deduce}), we get that $t_{\delta,y}>1$.
From (\ref{tdeltaypsi}) and the H\"{o}lder inequality, it follows that
\begin{align}\label{holder}
t_{\delta,y}^{2_{s}^{\ast}-2}\|\psi_{\delta,y}\|_{2_{s}^{\ast}}^{2_{s}^{\ast}}&=\|\psi_{\delta,y}\|_{D^{s,2}}^{2}+\int_{\mathbb{R}^{3}}V(x)\psi_{\delta,y}^{2}dx
+t_{\delta,y}^{2}\mathcal{N}(\psi_{\delta,y})\nonumber\\
&<t_{\delta,y}^{2}\|\psi_{\delta,y}\|_{D^{s,2}}^{2}+t_{\delta,y}^{2}\int_{\mathbb{R}^{3}}V(x)\psi_{\delta,y}^{2}dx
+t_{\delta,y}^{2}\mathcal{N}(\psi_{\delta,y})\nonumber\\
&\leq t_{\delta,y}^{2}\|\psi_{\delta,y}\|_{D^{s,2}}^{2}+t_{\delta,y}^{2}\|V\|_{\frac{3}{2s}}\|\psi_{\delta,y}\|_{2_{s}^{\ast}}^{2}\nonumber\\
&\quad+t_{\delta,y}^{2}S_{s}^{-\frac{1}{2}}\|K\|_{\frac{6}{6s-3}}\|\Phi(\psi_{\delta,y})\|_{D^{s,2}}\|\psi_{\delta,y}\|_{2_{s}^{\ast}}^{2}.
\end{align}
By (\ref{dt2}) and (\ref{phisu}),
\begin{align}\label{Phiu}
\|\Phi(u)\|_{D^{s,2}}\leq S_{s}^{-\frac{3}{2}}\|K\|_{\frac{6}{6s-3}}\|u\|_{D^{s,2}}^{2}.
\end{align}
By (\ref{deduce}), (\ref{Phiu}) and $\|\psi_{\delta,y}\|_{2_{s}^{\ast}}^{2_{s}^{\ast}}=\|\psi_{\delta,y}\|_{D^{s,2}}^{2}=S_{s}^{\frac{3}{2s}}$, we have
\begin{align}\label{upp}
t_{\delta,y}^{2_{s}^{\ast}-4}<1+S_{s}^{-1}\|V\|_{\frac{3}{2s}}+S_{s}^{\frac{3}{2s}-3}\|K\|_{\frac{6}{6s-3}}^{2}.
\end{align}
Then, by  (\ref{corresponding}), (\ref{tdeltaypsi}), (\ref{upp}) and (\ref{svk}),
\begin{align*}
I(t_{\delta,y}\psi_{\delta,y})&=\frac{1}{2}t_{\delta,y}^{2}\|\psi_{\delta,y}\|_{D^{s,2}}^{2}+\frac{1}{2}t_{\delta,y}^{2}\int_{\mathbb{R}^{3}}V(x)\psi_{\delta,y}^{2}dx
+\frac{1}{4}t_{\delta,y}^{4}N(\psi_{\delta,y})-\frac{1}{2_{s}^{\ast}}t_{\delta,y}^{2_{s}^{\ast}}\|\psi_{\delta,y}\|_{2_{s}^{\ast}}^{2_{s}^{\ast}}\\
&=\frac{s}{3}t_{\delta,y}^{2}\|\psi_{\delta,y}\|_{D^{s,2}}^{2}+\frac{s}{3}t_{\delta,y}^{2}\int_{\mathbb{R}^{3}}V(x)\psi_{\delta,y}^{2}dx
+\frac{4s-3}{12}t_{\delta,y}^{4}N(\psi_{\delta,y})\\
&\leq \frac{s}{3}t_{\delta,y}^{2}(\|\psi_{\delta,y}\|_{D^{s,2}}^{2}+\int_{\mathbb{R}^{3}}V(x)\psi_{\delta,y}^{2}dx+t_{\delta,y}^{2}N(\psi_{\delta,y}))\\
&=\frac{s}{3}t_{\delta,y}^{2_{s}^{\ast}}\|\psi_{\delta,y}\|_{2_{s}^{\ast}}^{2_{s}^{\ast}}\\
&<\frac{s}{3}S_{s}^{\frac{3}{2s}}(1+S_{s}^{-1}\|V\|_{\frac{3}{2s}}+S_{s}^{\frac{3}{2s}-3}\|K\|_{\frac{6}{6s-3}}^{2})^{\frac{2_{s}^{\ast}}{2_{s}^{\ast}-4}}\\
&\leq \frac{2s}{3}S_{s}^{\frac{3}{2s}}.
\end{align*}
Therefore, $\frac{s}{3}S_{s}^{\frac{3}{2s}}<d<\frac{2s}{3}S_{s}^{\frac{3}{2s}}$. By Corollary \ref{und}, $d$ is the critical value of $I$. The proof is complete.

\section{\textbf{Acknowledgements}}
This work is partially supported by National Natural Science Foundation of China under the contract
No.11571269, China Postdoctoral Science Foundation Funded Project under contracts No.2015M572539 and No.2016T90899 and Shaanxi Province Postdoctoral Science Foundation Funded Project.

\end{document}